\documentclass[reqno]{amsart}
\usepackage{amssymb,amscd,verbatim, amsthm,graphics, color,latexsym,amsmath,multicol}
\usepackage{fancybox}

\usepackage[top=3cm, bottom=3cm, left=3.2cm, right=3.2cm]{geometry}

\newcommand \fk[1]{{{\mathfrak #1}}}
\newcommand \C[1]{{\mathcal #1}}

\newcommand \wti[1]{{\widetilde {#1}}}

\newcommand\fg{\mathfrak g}

\newcommand \bC{{\mathbb C}}

\newcommand \bH{{\mathbb H}}
\newcommand \bR{{\mathbb R}}
\newcommand \bZ{{\mathbb Z}}

\newcommand \bQ{{\mathbb Q}}

\newcommand\CO{{\C O}}

\newcommand\al{{\alpha}}

\newtheorem{theorem}{Theorem}[section]

\newtheorem{corollary}[theorem]{Corollary}
\newtheorem{lemma}[theorem]{Lemma}
\newtheorem{proposition}[theorem]{Proposition}

\newtheorem{remark}[theorem]{Remark}

\newcommand\Hom{\operatorname{Hom}}

\newcommand\Ind{\operatorname{Ind}}

\newcommand\tr{\operatorname{tr}}

\newcommand\triv{\mathsf{triv}}
\newcommand\sgn{\mathsf{sgn}}
\newcommand\refl{\mathsf{refl}}

\newcommand\Irr{\mathsf{Irr}}

\newcommand\el{\mathsf{ell}}

\newcommand\sr{\mathsf{sr}}
\newcommand\ssr{\mathsf{ssr}}

\newcommand\sol{\mathsf{sol}}

\newcommand\gen{\mathsf{gen}}

\newcommand\Pin{\mathsf{Pin}}

\newcommand\Good{\mathsf{Good}}
\newcommand\rk{\mathsf{rk}}
\newcommand\Sg{\mathsf{Sg}}

\def\<{\langle} 
\def\>{\rangle}

\numberwithin{equation}{section}

\begin{document}

\title[Weyl groups and the Dirac inequality]{Weyl groups, the Dirac inequality, and isolated unitary unramified representations}

\author{Dan Ciubotaru}
        \address[D. Ciubotaru]{Mathematical Institute\\ University of
          Oxford\\ Oxford, OX2 6GG, UK}
        \email{dan.ciubotaru@maths.ox.ac.uk}

\begin{abstract}
I present several applications of the Dirac inequality to the determination of isolated unitary representations and associated ``spectral gaps" in the case of unramified principal series. The method works particularly well in order to attach irreducible unitary representations to the large nilpotent orbits (e.g., regular, subregular) in the Langlands dual complex Lie algebra. The results could be viewed as a $p$-adic analogue of Salamanca-Riba's classification of irreducible unitary $(\mathfrak g,K)$-modules with strongly regular infinitesimal character.
\end{abstract}


\maketitle

\setcounter{tocdepth}{1}

\section{Introduction}

\subsection{}Let $\mathsf k$ be a nonarchimedean local field with finite residue field of cardinality $q$ and valuation $\mathsf{val}_{\mathsf k}$. Let $G$ be the $\mathsf k$-points of a $\mathsf k$-split reductive group whose Langlands dual is the complex reductive group $G^\vee$. Let $T$ be a maximal $\mathsf k$-split torus in $G$, $\C R=( X^*(T),\Phi,X_*(T),\Phi^\vee)$ the corresponding root datum, and $W$ the finite Weyl group. Let $\langle~,~\rangle$ denote the natural pairing between $X_*(T)$ and $X^*(T)$. Fix a Borel subgroup $B\supset T$ of $G$ and let $\Phi^+$ and $\Pi$ be the subset of positive roots and simple roots, respectively. Set $T^\vee=X^*(T)\otimes_\bZ \bC^\times$, a maximal torus in $G^\vee$ with Lie algebra $\fk t^\vee=X^*(T)\otimes_\bZ \bC$, a Cartan subalgebra of the Lie algebra $\fg^\vee$ of $G^\vee$. Let $T^\vee_\bR=X^*(T)\otimes_\bZ \bR_{>0}$ and $\mathfrak t_\bR^\vee=X^*(T)\otimes_\bZ \bR$. Consider the group homomorphism
\[\textsf{val}_T: T\to X_*(T),\quad \langle \mathsf{val}_T(t),\lambda\rangle=\mathsf{val}_{\mathsf k}(\lambda(t)), \text{ for all }\lambda\in X^*(T).
\]
For every $s\in T^\vee=\Hom(X_*(T),\bC^\times)$, define the unramified character $\chi_s\in X^*(T)$ by $\chi_s=s\circ \textsf{val}_T.$ Concretely, if $\varpi_{\mathsf k}$ is a uniformiser of $\mathsf k$, then 
\[\chi_s(\varphi(\varpi_{\mathsf k}))=\varphi(s),\text{ for all }\varphi\in X_*(T)=X^*(T^\vee).
\]
For every $s\in T^\vee$, inflate $\chi_s$ to a character of $B$ (by making it trivial on the unipotent radical of $B$) and define the unramified principal series 
$X(s)=i_B^G(\chi_s)$; here $i$ denotes the functor of normalised parabolic induction. Then $X(s)$ has a finite composition series which only depends on the $W$-orbit of $s$. I will be interested in the unitarisability question for the composition factors of $X(s)$ in the case when $s\in T^\vee_{\bR}$, more precisely, when
\begin{equation}
s=\exp(\log(q)\nu),\text{ for some }\nu\in \mathfrak t^\vee_\bR.
\end{equation}
Write $X(\nu)$ in place of $X(s)$ in this case. Denote
\begin{equation}
G^\vee(\nu)=\{g\in G^\vee\mid \operatorname{Ad}(g)\nu=\nu\},\quad \fg^\vee_1(\nu)=\{x\in\fg^\vee\mid [\nu,x]=x\}.
\end{equation}
The group $G^\vee(\nu)$ acts on $\fg^\vee_1(\nu)$ via the adjoint action with finitely many orbits. By the results of Kazhdan and Lusztig \cite{KL,L2}, the composition factors of $X(\nu)$ are parameterised by pairs $(\C C,\C L)$, where 
$\C C$  is a $G^\vee(\nu)$-orbit on $\fg^\vee_1(\nu)$ and $\C L$ is an irreducible $G^\vee(\nu)$-equivariant local system on $\C C$ of Springer type. Denote the corresponding irreducible factor of $X(\nu)$ by $V(\C C,\C L)$. I am interested in the question: {\it for which $(\C C, \C L)$ is $V(\C C,\C L)$ unitarisable?}

There is a simpler, but equivalent, setting. Lusztig \cite{L1} defined an associative algebra, the graded affine Hecke algebra, $\bH$ from the root datum $\C R$ (in fact, it only depends on the root system). Let $\mathfrak t=X_*(T)\otimes_\bZ \bC$ and $\mathsf T^\bullet(\mathfrak t)$ be the tensor algebra on $\mathfrak t$. The algebra $\bH$ is the quotient of the tensor product algebra $\bC[W]\otimes \mathsf T^\bullet(\mathfrak t)$ by the relations
\begin{equation}
s_\alpha\cdot \omega-s_\alpha(\omega)\cdot s_\alpha= \langle \omega,\alpha\rangle,\text{ for all }\omega\in \mathfrak t,\ \alpha\in\Pi.
\end{equation}
Here $s_\alpha$ denotes the reflection for the root $\alpha$. The algebra $\bH$ has a natural filtration where $w\in W$ gets degree $0$ and $\omega\in\fk t$ degree 1, for which the associated graded algebra is $S(\fk t)\rtimes \bC[W]$. The centre of $\bH$ is $S(\mathfrak t)^W$ and so, the central characters of $\bH$ are parameterised by $W$-orbits in $\mathfrak t^\vee.$ By \cite{L2}, the simple $\bH$-modules with central character $\nu\in\mathfrak t^\vee_\bR$ are parameterised by the same data $(\C C,\C L)$ as before. Denote by $\overline V(\C C, \C L)$ the corresponding simple $\bH$-module. All simple $\bH$-modules are finite dimensional since $\bH$ has finite rank over its centre. 

The algebra $\bH$ has a star operation $*$ (e.g., \cite{BM2},\cite{Op}) with respect to which one defines unitary $\bH$-modules:
\begin{equation}
w^*=w^{-1},\quad \omega^*=-\omega+\sum_{\alpha\in\Phi^+} \langle \omega,\alpha\rangle s_\alpha,\quad w\in W,\ \omega\in\fk t_\bR.
\end{equation}

 The main result of \cite{BM1} (see \cite{Ci3} for a generalisation) can be rephrased as follows. 

\begin{theorem}[Barbasch-Moy]\label{t:BM} The $G$-representation $V(\C C,\C L)$ is unitary if and only if the $\bH$-module $\overline V(\C C,\C L)$ is unitary.
\end{theorem}

The unitarisability question can therefore be translated to the algebraic setting of finite-dimensional $\bH$-modules. In this setting, a pleasant tool for showing that modules are {\it not} unitary is the Dirac inequality inspired by \cite{Pa,HP}, proved in \cite{BCT}, and recalled below, in section \ref{s:dirac}.  We already know that the general shape of the unitary dual is complicated, see for example, \cite{Ta, Ba, Ci-F4, BC-E8}, in ways that can't be fully captured by the Dirac inequality alone. But of particular interest in the theory are the irreducible unitary representations which are isolated in the unitary dual \cite{Vo}, and, in this paper, I try to exemplify how one can use the Dirac inequality in its stronger form, Theorem \ref{dirac-ineq}, in order to determine certain isolated unitary representations and associated  spectral gaps for the parameter $\nu$. 

\smallskip

The results below are analogous to the classification of irreducible unitary $(\mathfrak g,K)$-modules with strongly regular infinitesimal character \cite{SR} (see also \cite{SRV}), not just in spirit, but also in terms of the general method used for proving non-unitarity: Parthasarathy's Dirac inequality.

\subsection{}I present the main result. Let $\C N^\vee$ denote the nilpotent cone in $\fg^\vee$. As it is well known, there exists a unique open dense $G^\vee$-orbit in $\C N^\vee$, the regular orbit $\CO^\vee_{\mathsf r}$. If $\fg^\vee$ is simple, there exists a unique open dense orbit in $\C N^\vee\setminus \CO^\vee_{\mathsf r}$, the subregular orbit $\CO^\vee_{\mathsf{sr}}$. If in addition, $\fg^\vee$ has rank at least 4 and it is not of type $C_n$ or $D_n$, there exists a unique open dense orbit in $\overline \CO^\vee_{\mathsf{sr}}\setminus  \CO^\vee_{\mathsf{sr}}$, that one might call the sub-subregular orbit $\CO^\vee_{\mathsf{ssr}}$. For every nilpotent orbit $\CO^\vee$, fix a representative $e^\vee \in \CO^\vee$ and a Lie triple $\{e^\vee, h^\vee, f^\vee\}$, such that $h^\vee=h^\vee_{\CO^\vee}\in\fk t^\vee_\bR$. 

There is an involution on the set of smooth irreducible representations of $G$ defined by Zelevinsky and, in general, by Aubert  \cite{Au}. In the correspondence from Theorem  \ref{t:BM}, this involution translates to the Iwahori-Matsumoto involution for affine Hecke algebras which is easily seen to map unitary modules to unitary modules. 

The Dirac inequality can be used efficiently in order to prove the following:
 
 \begin{theorem}\label{t:main}
Suppose $G$ is a simple split $p$-adic group of rank at least $2$. Let $\nu\in \mathfrak t^\vee_\bR$ be given.
\begin{enumerate}
\item If  $|\nu|> |h^\vee_{\mathsf{sr}}/2|$  then the only unitary subquotients of $X(\nu)$ are the trivial and the Steinberg representations, which occur when $\nu\in W\cdot h^\vee_{\mathsf{r}}/2$.
\item If $|\nu|=|h^\vee_{\mathsf{sr}}/2|$, then $X(\nu)$ has a unitary subquotient if and only if $\nu\in W\cdot h^\vee_{\mathsf{sr}}/2$. Suppose this is the case, then:
\begin{enumerate}
\item if $\rk(\fg^\vee)=2$, every subquotient of $X(h^\vee_{\mathsf{sr}}/2)$ is unitary;
\item if $\rk(\fg^\vee)=3$, the unitary subquotients of $X(h^\vee_{\mathsf{sr}}/2)$ are given in section \ref{s:5}, see Theorem \ref{t:unitary};
\item if $\rk(\fg^\vee)\ge 4$, the only unitary subquotients of $X(h^\vee_{\mathsf{sr}}/2)$ are the irreducible tempered ones and their Aubert-Zelevinsky duals, except when $G=SO(2n+1)$  where the additional unitary subquotients are listed in Theorem \ref{t:unitary}.
\end{enumerate}
\item If $G$ has rank at least $4$ and it is not a special orthogonal group, then
\begin{enumerate}
\item if $|h^\vee_{\mathsf{ssr}}/2|<|\nu|<|h^\vee_{\mathsf{sr}}/2|$, no subquotient of $X(\nu)$ is unitary;
\item if $|\nu|=|h^\vee_{\mathsf{ssr}}/2|$, then $X(\nu)$ has unitary subquotients if and only if $\nu\in W\cdot h^\vee_{\mathsf{ssr}}/2$;
\item the subquotients of $X(h^\vee_{\mathsf{ssr}}/2)$ that are unitary are given in Theorem \ref{t:unitary}.
\end{enumerate}
\end{enumerate}

\end{theorem}

The fact that the only irreducible unitary representations with parameter $h^\vee_\mathsf{r}/2$ are the trivial and the Steinberg representations has been known by Casselman's work \cite{Cas}. That this is also a consequence of the Dirac inequality has already been noted in \cite{BCT}. The other results in Theorem \ref{t:main} are new as far as I know. I emphasise that at each one of the parameters $\nu=h^\vee_{\mathsf r}/2$, $h^\vee_\sr/2$, $h^\vee_\ssr/2$, the number of non-isomorphic simple subquotients of $X(\nu)$ is exponential in $\rk(\fg^\vee)$, so Theorems \ref{t:main} and \ref{t:unitary} say that very few of them are unitary (not more than $8$).

\smallskip

As it can be easily appreciated, the Dirac inequality method works well when the parameter $\nu$ is attached to large nilpotent orbits. At the other extreme, if $\nu$ is small,  the Dirac inequality does not give any information. More precisely, I show that there is a unique nilpotent orbit $\CO^\vee_{\min}$ such that the Dirac inequality is vacous when $|\nu|\le |h^\vee_{\min}/2|$. (The only new observation here is the {\it uniqueness} of this orbit in relation to the other nilpotent orbits that appear in the Dirac inequality, Proposition \ref{p:O-min}.) 
This makes the parameter $\nu\in W\cdot h^\vee_{\min}/2$ particularly interesting. To exemplify this, I study the example of $F_4$, where $\CO^\vee_{\min}=F_4(a_3)$, and one can deduce from \cite{Ci,Ci-F4} that there are $19$ non-isomorphic irreducible subquotients of $X(h^\vee_{\min}/2)$, exactly $18$ of which are unitary (see Proposition \ref{th:F4}). 

\medskip

\noindent{\bf Acknowledgements.} I thank the referee for several useful comments, in particular for the analogy with the real groups case. This research was supported in part by the EPSRC grants EP/N033922/1 (2016) and EP/V046713/1 (2021).

\section{The Dirac inequality}\label{s:dirac}

Fix  a $W$-invariant positive definite symmetric form $(~,~)$ on $\fk t^\vee_\bR$, extend it to a symmetric bilinear form on $\fk t^\vee$. Denote $|a|=(a,a)^{1/2}$ for every $a\in \fk t^\vee_\bR$. Let $C(\fk t^\vee)$ denote the Clifford algebra defined by $\fk t$ and $(~,~)$ and relations
\[a\cdot a'+a'\cdot a=-2(a,a'),\quad a,a'\in\fk t^\vee.
\]
The algebra $C(\fk t^\vee)$ has a $\bZ/2$-grading given by the parity of the degree and a $\bZ$-filtration by the degree ($\fk t^\vee$ gets degree $1$) with associated graded algebra $\bigwedge\fk t^\vee$. Define $\epsilon$ to be the identity of the even part $C(\fk t^\vee)_0$ of $C(\fk t^\vee)$ and the negative of the identity on the odd part. The transpose anti-automorphism $x\to x^t$ of $C(\fk t^\vee)$ is defined to be the identity on $\fk t^\vee$. Then the pin group is
\[\Pin(\fk t^\vee)=\{x\in C(\fk t^\vee)^\times\mid \epsilon(x) \omega x^t\in \fk t^\vee,\text{ for all }\omega\in\fk t^\vee,\ x^t=x^{-1}\}.
\]
This is a two-fold cover of the orthogonal group $O(\fk t^\vee)$ via the projection map
\[x\mapsto (p(x):\fk t^\vee\to \fk t^\vee,\ p(x)(\omega)=\epsilon(x) \omega x^t.
\]
The kernel of $p$ is $\{1,z\}\cong C_2$, where $z$ is the negative of the identity element in $C(\mathfrak t^\vee)$. Define the pin cover of the Weyl group
\begin{equation}
\wti W=p^{-1}(W)\subset \Pin(\fk t^\vee),
\end{equation}
a two-fold cover of $W$. It is generated by the order $4$ elements
\[
\wti s_\al=\frac {\alpha}{|\alpha|},\quad \alpha^\vee\in\Phi^{+}.
\]
The following central element of $\bC[\wti W]$ plays an important role \cite{BCT}:
\begin{equation}
\Omega_{\wti W}=\frac{z}{4}\sum_{\alpha,\beta\in\Phi^+,~s_\alpha(\beta)\in \Phi^-}|\alpha| |\beta| \wti s_\alpha \wti s_\beta.
\end{equation}
Let $\Irr_\gen\wti W$ denote the set of (equivalence classes of) irreducible $\wti W$-representations in which $z$ acts by the negative of the identity (the genuine representations). If $\wti \sigma\in \Irr_\gen\wti W$, let $N(\wti\sigma)$ denote the scalar by which $\Omega_{\wti W}$ acts in $\wti \sigma$. 

If $\dim \fk t^\vee$ is even, $C(\fk t^\vee)$ is a central simple algebra with a unique complex simple module $S$ (the space of spinors). When $\dim\fk t^\vee$ is odd, the even part $C(\fk t^\vee)_0$ is central simple and its unique simple module can be extended to $C(\fk t^\vee)$ in two inequivalent ways, $S^+$ and $S^-$. To make the notation uniform, let $\C S=S$ in the even case, and $\C S=S^++S^-$ in the odd case.  If the root datum $\C R$ is semisimple, the restriction of $S$, or $S^+$, $S^-$ to $\wti W$ is irreducible and 
\[ N(S)=(\rho,\rho),\text{ where }\rho=\frac 12\sum_{\alpha\in\Phi^+} \alpha\in \fk t^\vee.
\]
It was observed and proved case-by-case in \cite{Ci}, then uniformly and sharpened in \cite{CH}, that there exists a close relation between $\Irr_\gen \wti W$ and the nilpotent cone of $\fg^\vee$ and Springer's correspondence. Let $\C N^\vee_\sol$ denote the set of nilpotent elements $e^\vee$ of $\fg^\vee$ such that the centralizer of $e^\vee$ in $\fg^\vee$ is a solvable algebra. The set $\C N^\vee_\sol$ is a union of adjoint nilpotent orbits in $\fg^\vee$. The classification of these orbits for every simple Lie algebra is given in \cite{CH}. Recall that Springer's correspondence is a bijection between $\Irr W$ and the set $\{(\CO^\vee,\C E)\}$, $\CO^\vee$ a nilpotent $G^\vee$-orbit in $\fg^\vee$  and $\C E$ an irreducible $G^\vee$-equivariant local system on $\CO^\vee$ of Springer type. Denote by $\sigma(\CO^\vee,\C E)$ the corresponding irreducible $W$-representation. To make the relation with the $W$-structure of standard modules for the Hecke algebra more convenient, I normalise this correspondence so that $\sigma(0,\triv)=\triv$. Then the observation is that there exists a natural surjective map
\begin{equation}\label{Theta}
\Theta: \Irr_\gen \wti W\to \C N^\vee_\sol
\end{equation}
such that if $\Theta(\wti\sigma)=\CO^\vee$, then
\begin{enumerate}
\item there exists a local system $\C E$ on $\CO^\vee$ such that $\Hom_{\wti W}[\wti\sigma,\sigma(\CO^\vee,\C E)\otimes \C S]\neq 0$;
\item $N(\wti \sigma)=(h^\vee/2,h^\vee/2)$ where $h^\vee=h^\vee_{\CO^\vee}$.
\end{enumerate}
For more precise properties of $\Theta$, see \cite[Theorems 6.5, 6.6, Corollary 6.4]{CH}. In \cite{Ci} and in \cite[Appendix A]{CH}, $\Theta$ is described explicitly for each simple root system.

\begin{theorem}[The extended Dirac inequality \cite{BCT}]\label{dirac-ineq} Suppose $\overline V$ is a simple $\bH$-module admitting a nonzero $*$-hermitian invariant form. Let $\nu\in\fk t^\vee_\bR$ be (a representative of) the central character of $\overline V$. If $\overline V$ is unitary then
\[|\nu|\le |h^\vee_{\CO^\vee}/2|,
\]
for all $\CO^\vee\in G^\vee\backslash \C N^\vee_\sol$ such that there exists 
\begin{equation}\label{e:cond-D}
\wti\sigma\in \Theta^{-1}(\CO^\vee)\text{ with }\Hom_{\wti W}[\wti\sigma, \overline V|_W\otimes\C S]\neq 0.
\end{equation} 
Moreover, if $\overline V$ is unitary and $|\nu|= |h^\vee_{\CO^\vee}/2|$ such that (\ref{e:cond-D}) is satisfied, then $\nu\in W\cdot h^\vee/2$.
\end{theorem}

The second part of the result is a consequence of Dirac cohomology and ``Vogan's conjecture" for $\bH$-modules \cite{BCT}.

\section{Recollections about spin representations of the Weyl group}

I use \cite{Ci4} and \cite{CH}, where a new interpretation of the irreducible $\wti W$-representations is given together with the connections with Springer theory. Originally, the irreducible $\wti W$-representations had been classified by Schur, Morris, and Read.

Let $P(n)$ denote the set of partitions of $n$ and $DP(n)$ the subset of partitions into distinct parts. For example, $|P(5)|=7$ and $|DP(5)|=3$. If $\lambda\in P(n)$, denote the transpose partition by $\lambda^T$. For $\lambda\in P(n)$, let $\ell(\lambda)$ be the number of parts of $\lambda$. If $\lambda$ is a partition, let $|\lambda|$ denote the sum of its parts.

For a Weyl group $W$, let $\triv$, $\refl$, and $\sgn$ denote the trivial, reflection, and sign representations, respectively. Let $R(W)$ denote the Grothendieck group of $W$ and $\langle~,~\rangle_W$ the character pairing.

Set $a(\fg^\vee)=1$ when $\rk(\fg^\vee)$ is even and $a(\fg^\vee)=2$ when $\rk(\fg^\vee)$ is odd. Then
\[\C S\otimes \C S=a(\fg^\vee) {\bigwedge}\fk t \text{ as $W$-representations}.
\]
This is the reason why it will be convenient to look at character values on elements $w\in W$, such that $\det_{\fk t}(1+w)\neq 0$, where $\det_{\fk t}(1+w)=\tr_{{\bigwedge}\fk t} (w)$. We call these elements $w\in W$ $(-1)-elliptic$. The support of $\C S$ is therefore precisely the preimage in $\wti W$ of the set of $(-1)-elliptic$ elements of $W$.
Notice that if $W_J$ is a parabolic subgroup of $W$ and $w\in W_J$ is $(-1)$-elliptic in $W_J$, then $w$ is $(-1)$-elliptic in $W$.

\smallskip

The irreducible $S_n$-representations are in one-to-one correspondence with $P(n)$, via Young diagrams. Denote $\sigma_\lambda$ the irreducible $S_n$-representation corresponding to $\lambda$. In this notation, $\sigma_{(n)}=\triv$, $\sigma_{(n-1,1)}=\refl$, and $\sigma_{(1^n)}=\sgn$, and $\sigma_\lambda\otimes\sgn=\sigma_{\lambda^T}$.

The Weyl group of type $B_n/C_n$ is $W(B_n)=S_n\ltimes (\bZ/2)^n$. Its irreducible representations are constructed via Mackey induction/restriction and they are in one-to-one correspondence with the set $BP(n)$ of bipartitions of $n$, i.e., pairs of partitions $(\al,\beta)$ such that $|\alpha|+|\beta|=n$. Given a bipartition $(\al,\beta)$, let $\alpha\times\beta$ denote the corresponding irreducible $W(B_n)$-representation. In this notation, $(n)\times\emptyset=\triv$, $(n-1)\times (1)=\refl$, $\emptyset\times (1^n)=\sgn$, and $(\alpha\times\beta)\otimes\sgn=\beta^T\times\alpha^T$.

The Weyl group of type $D_n$ is a subgroup of $W(B_n)$ of index two. Its irreducible representations are obtained by restriction from $W(B_n)$. If $\alpha\neq\beta$, then $\alpha\times\beta$ restricts irreducibly to $W(D_n)$ and $\alpha\times\beta$ and $\beta\times\alpha$ restrict to the same $W(D_n)$-representation. If $\alpha=\beta$, then $(\alpha\times\alpha)|_{W(D_n)}=(\alpha\times\al)^+\oplus (\alpha\times\al)^-$, and $(\al\times\al)^\pm$ are two inequivalent irreducible $W(D_n)$-representations.

For the Weyl groups of exceptional types, I use Carter's notation for irreducible characters \cite{Ca}.

\

If $e^\vee\in \C N^\vee$ and $\C B_{e^\vee}$ is the corresponding Springer fibre, every cohomology group $H^i(\C B_{e^\vee},\bC)$ carries a representation of $W$ \cite{Sp} which commutes with the action of the group of components $A(e^\vee)$ of the centraliser of $e^\vee$ in $G^\vee$. The equivariant local systems on $\CO^\vee=G^\vee\cdot e^\vee$ are in bijection with $\Irr A(e^\vee)$. Fix $\C E$ a local system of Springer type on $\CO^\vee$ and let $\phi\in \Irr A(e^\vee)$ be the corresponding representation. Set
\[
X_q(\CO^\vee,\C E)=\sum_{i=0}^{d_{e^\vee}} q^{d_{e^\vee}-i} H^{2i}(\C B_{e^\vee},\bC)^\phi\otimes\sgn,
\]
a $q$-graded $W$-representation. Here $2d_{e^\vee}$ is the real dimension of $\C B_{e^\vee}$. Notice that \[X_q(\CO^\vee,\C E)\in\sigma(\CO^\vee,\C E)+ q \bZ[q]\otimes_\bZ R(W).\]
By a result of Borho and Macpherson, the only irreducible $W$-representations which occur in $X_q(\CO^\vee,\C E)$, other than $\sigma(\CO^\vee,\C E)$ itself, are of the form $\sigma(\CO'^\vee,\C E')$ with $\CO'^\vee\subset \overline{\CO^\vee}\setminus \CO^\vee$. Put a total order on the set $\{(\CO,\C E)\}$ which is compatible with the closure ordering of nilpotent orbits. This means that, with respect to this order, the change of basis matrix $P(q)$ between the sets $\{X_q(\CO^\vee,\C E)\}_{(\CO^\vee,\C E)}$ and $\{\sigma(\CO^\vee,\C E)\}_{(\CO^\vee,\C E)}$ is upper triangular with $1$ on the diagonal and off-diagonal entries in $q\bZ_{\ge 0}[q]$. Hence the inverse matrix 
\begin{equation}
Q(q)=P(q)^{-1}
\end{equation}
is upper triangular with $1$ on the diagonal and off-diagonal entries in $q\bZ[q]$. 

I will be mainly interested in the specialisation  $X_{-1}(\CO^\vee,\C E)$. Define the (a priori virtual) genuine character of $\wti W$,
\begin{equation}
\wti\tau(\CO^\vee,\C E)=X_{-1}(\CO^\vee,\C E)\otimes \C S.
\end{equation}

\begin{theorem}[\cite{CH}]\label{t:CH}
\begin{enumerate}
\item Given a pair $(\CO^\vee,\C E)$, $\wti\tau(\CO^\vee,\C E)=0$, unless $\CO^\vee\in G^\vee\backslash \C N^\vee_\sol$. 
\item If $\CO^\vee\in G^\vee\backslash \C N^\vee_\sol$ and $\C E$ is of Springer type, then $\wti\tau(\CO^\vee,\C E)$ is the (nonzero) character of a  genuine $\wti W$-representation.
\item The map $\Theta$ in (\ref{Theta}) is defined by: $\Theta(\wti\sigma)=\CO^\vee$ if $\Hom_{\wti W}[\wti\sigma, \wti\tau(\CO^\vee,\C E)]\neq 0$ for some $\C E$.
\end{enumerate}
\end{theorem}

For example, when $\fg^\vee=\mathfrak{sl}(n,\bC)$, there is a one-to-one correspondence between $G^\vee\backslash \C N^\vee_\sol$ and $DP(n)$ given by the Jordan normal form. Set $a_\lambda=2^{\ell(\lambda)/2}$ when both $n$ and $\lambda$ are even, and $a_\lambda=2^{\lfloor(\ell(\lambda)-1)/2\rfloor}$, otherwise. If $\CO^\vee_\lambda\in G^\vee\backslash \C N^\vee_\sol$, $\lambda\in DP(n)$, the only irreducible equivariant local system is trivial and I drop it from notation. For $\lambda\in DP(n)$, it is shown in \cite{CH}, that
\begin{equation}
\wti\sigma_\lambda=\frac 1{a_\lambda} \wti\tau(\CO^\vee_\lambda)
\end{equation}
is an irreducible $\sgn$-self-dual $\wti S_n$-representation if $\lambda$ is an even partition, and $\wti\sigma_\lambda=\wti\sigma_\lambda^++\wti\sigma_\lambda^-$ if $\lambda$ is odd, where $\wti\sigma_\lambda^\pm$ are irreducible and $\sgn$-dual $\wti S_n$-representations. 

In particular, if $\mu\in P(n)$ and $\lambda\in DP(n)$, 
\begin{equation}
\Hom_{\wti S_n}[\sigma_\mu\otimes \C S,\wti\sigma_\lambda]=\frac {a(\mathfrak{sl}(n))}{a_\lambda} \Hom_{S_n}[\sigma_{\mu}, X_{-1}(\CO_\lambda^\vee)\otimes\bigwedge \fk t].
\end{equation}
Recall that the closure ordering of nilpotent orbits in $\mathfrak{sl}(n)$ corresponds to the dominance ordering $\ge$ of partitions. Write, as we may,
\[\sigma_\mu=X_{-1}(\CO^\vee_\mu)+\sum_{\lambda'>\mu} Q(-1)_{(\mu,{\lambda'})} ~X_{-1}(\CO^\vee_{\lambda'}),\quad Q(-1)_{(\mu,{\lambda'})}\in\bZ.
\]
It is proved in \cite{CH} that 
\[\Hom_{S_n}[X_{-1}(\CO^\vee_{\lambda'}), X_{-1}(\CO_\lambda^\vee)\otimes\bigwedge \fk t]\neq 0
\]
if and only if $\lambda'=\lambda\in DP(n)$, and in that case, the dimension of this space is $2^{\ell(\lambda)-1}$. In conclusion, we have:

\begin{lemma}
If $\mu\in P(n)$ and $\lambda\in DP(n)$, $\langle \sigma_\mu\otimes \C S,\wti\sigma_\lambda\rangle_{\wti S_n}=\frac{a(A_{n-1})\cdot 2^{\ell(\lambda)-1}}{a(\lambda)} Q(-1)_{(\mu,{\lambda})}.$
\end{lemma}
Hence, this multiplicity is nonzero if and only if the $(-1)$-Kostka number $Q(-1)_{(\mu,{\lambda})}$ is nonzero. These numbers admit a combinatorial interpretation in terms of Young tableaux which I won't recall here, see \cite{St} for example.

\

The irreducible $\wti W(B_n)$-representations have a particularly simple form. 

\begin{proposition}[Read]
For every $\lambda\in P(n)$, the $\wti W(B_n)$-representations $(\lambda\times\emptyset)\otimes S$ (when $n$ is even) and $(\lambda\times\emptyset)\otimes S^\pm$ (when $n$ is odd) are irreducible. They form a complete set of inequivalent irreducible $\wti W(B_n)$-representations.

If $n$ is odd, every $(\lambda\times\emptyset)\otimes S^\pm$ restricts to an irreducible $\wti W(D_n)$-representation and these restrictions form a complete set of inequivalent irreducible $\wti W(D_n)$-representations.

If $n$ is even and $\lambda\neq\lambda^T$, then $(\lambda\times\emptyset)\otimes S$ and $(\lambda^T\times\emptyset)\otimes S$ restrict to the same irreducible $\wti W(D_n)$-representation. If $\lambda=\lambda^T$,  the restriction of $(\lambda\times\emptyset)\otimes S$ splits into a direct sum of two inequivalent $\sgn$-selfdual $\wti W(D_n)$-representations. These restrictions form  a complete set of inequivalent irreducible $\wti W(D_n)$-representations.
\end{proposition}

In this case, if $(\al,\beta)\in BP(n)$, and $\lambda\in P(n)$, 
\begin{equation}
\Hom_{\wti W(B_n)}[(\al\times\beta)\otimes \C S,(\lambda\times\emptyset)\otimes \C S]=a(B_n) \Hom_{W(B_n)}[\al\times\beta,(\lambda\times\emptyset)\otimes \bigwedge \fk t].
\end{equation}
It is easy to see that 
\[ \Hom_{W(B_n)}[\al\times\beta,(\lambda\times\emptyset)\otimes \bigwedge \fk t]=\Hom_{S_{|\al|}\times S_{|\beta|}}[\sigma_\al\boxtimes \sigma_{\beta^T},\sigma_\lambda].
\]
Thus the question whether or not a certain irreducible $\wti W(B_n)$-representation occurs in the tensor product $(\al\times\beta)\otimes \C S$ is equivalent to whether or not a Littlewood-Richardson coefficient is nonzero:
\begin{lemma}\label{l:spin-B}
If $(\al,\beta)\in BP(n)$, and $\lambda\in P(n)$, $\langle (\al\times\beta)\otimes \C S,(\lambda\times\emptyset)\otimes \C S\rangle_{\wti W(B_n)}=c_{\al,\beta^T}^\lambda$, where $c_{\al,\beta^T}^\lambda=\langle \sigma_\al\boxtimes \sigma_{\beta^T},\sigma_\lambda\rangle_{S_{|\al|}\times S_{|\beta|}}.$
\end{lemma}

If $\fg^\vee=\mathfrak{sp}(2n)$, the nilpotent orbits in $\C N^\vee_\sol$ are parameterised (in the analogue of the Jordan form) by partitions of $2n$ where all parts are even and each distinct part occurs with multiplicity at most $2$. If $\fg^\vee=\mathfrak{so}(m)$, the nilpotent orbits in $\C N^\vee_\sol$ are parameterised by partitions of $m$ where all parts are odd and each distinct part occurs with multiplicity at most $2$.

\section{Non-unitarity certificates}

I present several applications of Theorem \ref{dirac-ineq}. 

\subsection{}Fix $\nu\in \fk t^\vee_\bR$
and look at the simple $\bH$-modules $\overline V=\overline V(\C C',\C L')$ with central character $W\cdot \nu$. Let $\CO'^\vee$ be the $G^\vee$-saturation of $\C C'$ with Lie triple $(e'^\vee,h'^\vee,f'^\vee)$. 
Since $G^\vee(\nu)$ acts with finitely many orbits on $\fg_1(\nu)$, there exists a unique open dense orbit $\C C$ in $\fg_1(\nu)$. Denote the $G^\vee$-saturation of $\C C$ by $\C O^\vee$. In particular, this means that $\CO'^\vee\subset \overline{\CO^\vee}$. 

Let $\mathfrak z(\CO'^\vee)$ denote the reductive Lie algebra which is the centralizer of the triple $(e'^\vee,h'^\vee,f'^\vee)$ in $\fg^\vee$. The fact that $\overline V$ has central character $W\cdot h^\vee/2$ implies (see \cite{BC-E8}) that there exists $w\in W$ and $\nu_z\in\fk t^\vee_\bR\cap \mathfrak z(\CO'^\vee)$ such that
\[ w(\nu)/2=h'^\vee/2+\nu_z.
\]
In particular, $(h'^\vee,\nu_z)=0$. Hence 
\begin{equation}\label{e:lengths} |\nu|^2=|h'^\vee|^2/4+|\nu_z|^2.
\end{equation}

The following consequence was noted in \cite{BCT} in a somewhat different form. It says that there no complementary series from tempered modules attached to the orbits in $\C N^\vee_\sol$.

\begin{proposition} Let $\CO^\vee\subset\C N^\vee_\sol$ be given. Suppose $\overline V(\C C,\C L)$ is a simple $\bH$-module with central character $ \nu\in \fk t^\vee_\bR$ such that $\C C\subset \CO^\vee$. Then $\overline V(\C C,\C L)$ is unitary if and only if $\nu\in W\cdot h^\vee/2$ (in which case $\overline V(\C C,\C L)$ is tempered).

\end{proposition}

\begin{proof}
From (\ref{e:lengths}), write $w(\nu)/2=h^\vee/2+\nu_z$ for some $\nu_z\in \mathfrak z(\CO^\vee)$. By \cite{L2}, since $\C C\subset\CO^\vee$, there exists a local system $\C E$ of Springer type on $\CO^\vee$ such that 
\[\Hom_W[\sigma(\CO^\vee,\C E),\overline V(\C C,\C L)]\neq 0.
\]
By the properties of the map $\Theta$, it means that there exists $\wti\sigma\in\Theta^{-1}(\CO^\vee)$ which occurs in $\sigma(\CO^\vee,\C E)\otimes \C S$ and hence $\overline V(\C C,\C L)$ satisfies condition (\ref{e:cond-D}). Theorem  \ref{dirac-ineq} implies that if $\overline V(\C C,\C L)$ is hermitian, then it is unitary only if $|\nu|\le |h^\vee/2|$. But this implies that $\nu_z=0$ and therefore $\nu=w^{-1}(h^\vee/2)$.

Conversely, suppose $\overline V(\C C,\C L)$ has central character $h^\vee/2$. Since the $G^\vee$-saturation of $\C C$ is $\C O^\vee$, this means that $\C C$ is the unique open dense $G^\vee(h^\vee/2)$-orbit in $\fg^\vee_1(h^\vee/2)$. By \cite{KL,L2}, it follows that the module $\overline V(\C C,\C L)$ is tempered, hence unitary.
\end{proof}

\subsection{}Condition (\ref{e:cond-D}) can be used in order to prove non-unitarity as follows. Define a map
\begin{equation}\label{d:sigma}
d: \Irr W\to \mathbb N,\quad d(\sigma)=\min\{|h^\vee/2|~\mid \Hom_{\wti W}[\wti \sigma,\sigma\otimes\C S]\neq 0\text{ for some }\wti\sigma\in\Theta^{-1}(\CO^\vee)\}.
\end{equation}
The Dirac inequality says that if $\overline V$ is a unitary $\bH$-module with central character $\nu$, then 
\[
|\nu|\le \min\{\sigma\in\Irr W\mid \Hom_W[\sigma,\overline V]\neq 0\}.
\]
In light of this, it is particularly interesting to look at central characters $\nu$ such that $|\nu|=|h^\vee_{\CO^\vee}/2|$ for a fixed $\CO^\vee\subset \C N^\vee_\sol$. Define
\[\Good(\CO^\vee)=\{\sigma\in\Irr W\mid d(\sigma)\ge |h^\vee/2|\}. 
\]
Then clearly the Dirac inequality gives the following criterion:

\begin{corollary}\label{c:good}If $\overline V$ has central character $\nu$ with $|\nu|\ge |h^\vee/2|$ and $\overline V$ contains a $W$-type $\sigma$ such that $\sigma\notin \Good(\CO^\vee)$, then $\overline V$ is not unitary.
\end{corollary}

The philosophy is that the larger $\CO^\vee$ is, the fewer {\it good} $W$-types it admits and therefore, the fewer simple modules can be unitary.  Of course, if $\CO'^\vee\subset \overline \CO^\vee$, so that $|h'^\vee|\le |h^\vee|$, then
\[\Good(\CO'^\vee)\supseteq \Good(\CO^\vee).
\]

To compute $\Good(\CO^\vee)$ in certain cases, I will use the connection with the Springer correspondence and induction. In light of the results of \cite{CH}, see Theorem \ref{t:CH} above, define $\overline R(W)_{-1}$ to be the $(-1)$-elliptic representation space of $W$: the quotient of $R(W)\otimes_\bZ \bC$ by the radical of the $(-1)$-elliptic form
\[\langle\sigma,\sigma\rangle_{-1}^{\el}=\langle\sigma,\sigma'\otimes {\bigwedge}^\bullet \mathfrak t\rangle_W.
\]
This space can be identified with the space of class functions on $W$ supported on the $(-1)$-elliptic conjugacy classes. If $\sigma\in R(W)$, denote by $[\sigma]_{-1}$ its class in $\overline R(W)_{-1}$. To make the connection with Springer theory, recall that we may write 
\begin{equation}
\sigma(\CO_0^\vee,\C E_0)=X_q(\CO_0^\vee,\C E_0)+\sum_{(\CO'^\vee,\C E'),\ \CO'^\vee\subset \partial\overline{\CO_0^\vee}} Q(q)_{(\C E_0,\C E')}~X_q(\CO'^\vee,\C E')
\end{equation}
for certain polynomials $Q(q)_{(\C E,\C E')}\in q\bZ[q]$. Specialise at $q=-1$ and tensor with $\C S$:
\begin{equation}\label{e:-1}
\sigma(\CO_0^\vee,\C E_0)\otimes \C S=\sum_{(\CO'^\vee,\C E'),\  \CO'^\vee\subset \overline{\CO_0^\vee}\cap  \C N^\vee_\sol}Q(-1)_{(\C E_0,\C E')}~ \wti\tau(\CO'^\vee,\C E').
\end{equation}
Let $R(\wti W)_\gen$ denote the $\bC$-span of $\Irr_\gen\wti W$. This space has an involution coming from tensoring with the non-genuine sign representation, and let $R(\wti W)_\gen^\Sg$ denote the $(+1)$-eigenspace of this involution. Notice that $\C S\in R(\wti W)_\gen^\Sg$ and since the set $\Good(\CO^\vee)$ is closed under tensoring with the sign representation, we are really only interested in the basis elements of $R(\wti W)_\gen^\Sg$. By \cite[Proposition 6.1]{CH}, the map
\[\iota:\overline R(W)_{-1}\to R(\wti W)_\gen^\Sg,\quad \sigma\mapsto \sigma\otimes \C S
\]
is injective and its image is spanned by $\tau(\CO^\vee,\C E)$, $\CO^\vee\subset \C N^\vee_\sol$. Moreover, if $\fg^\vee$ is simple, then $\iota$ is a linear isomorphism unless $\fg^\vee$ is $E_7$ or $D_{2n}$. In $E_7$, the failure of surjectivity is because of the orbit $\CO^\vee=A_4+A_1$, which turns out to be the smallest orbit in $\C N^\vee_\sol$ in $E_7$. In $D_{2n}$, the failure of surjectivity is because of orbits parameterised by partitions $(a_1,a_1,a_2,a_2,\dots,a_k,a_k)$ of $2n$, where $a_i$ are distinct odd integers. The following proposition could be extended to cover these cases as well by making appropriate changes for these exceptions, but I will only apply it to do computations for large orbits, where these exceptions do not enter.

\begin{proposition}\label{p:span}Suppose $\fg^\vee$ is a simple Lie algebra.  
Let $\CO^\vee$ be an orbit in $\C N^\vee_\sol$. If $\fg^\vee$ is of type $E_7$, assume $\CO^\vee\neq A_4+A_1$ and if $\fg^\vee=\mathfrak {so}(2n)$, suppose $|h^\vee_{\CO^\vee}|>|h^\vee_{\CO^\vee_1}|$, where $\CO^\vee_1$ is parameterised by $(k+1,k+1,k-1,k-1)$ if $n=2k$ or $(k+2,k+2,k-1,k-1)$ if $n=2k+1$. 

An irreducible Weyl group representation $\sigma=\sigma(\CO_0^\vee,\C E_0)$ belongs to $\Good(\CO^\vee)$ if and only if in $\overline R(W)_{-1}$:
\[ [\sigma]_{-1}\in\text{span}\{[X_{-1}(\CO'^\vee,\C E')]_{-1}: \CO'^\vee\subset \overline \CO_0^\vee\cap \C N^\vee_\sol \text{ and } |h^\vee_{\CO'^\vee}|\ge |h^\vee_{\CO^\vee}|\}.
\]
\end{proposition}

\begin{proof}
From (\ref{e:-1}), we see that $\Hom_{\wti W}[\wti\tau(\CO'^\vee,\C E'),\sigma(\CO_0^\vee,\C E_0)\otimes \C S]\neq 0$ if and only if $Q(-1)_{(\C E_0,\C E')}\neq 0$. Moreover, this is equivalent with $\langle X_{-1}(\CO'^\vee,\C E'),\sigma(\CO_0^\vee,\C E_0)\rangle^\el_{-1}\neq 0$. The conditions on $\CO^\vee$ in types $E_7$ and $D_{2n}$ are such that no orbits of the exceptional kind that I mentioned above can have $ |h^\vee_{\CO'^\vee}|\ge |h^\vee_{\CO^\vee}|$. The claim then follows from the definition of $\Good(\CO^\vee)$.
\end{proof}

\begin{lemma}[\cite{BCT}]\label{l:reg}
 If $\CO^\vee_{\mathsf r}$ is the regular nilpotent orbit, then $\Good(\CO^\vee_{\mathsf r})=\{\triv,\sgn\}$.
\end{lemma}

\begin{proof}
This result has already appeared in \cite{BCT}, but I will reproduce the argument for convenience. By Proposition \ref{p:span}, a $W$-type $\sigma$ is good for $\CO^\vee_{\mathsf r}$ if and only if $[\sigma]_{-1}=m[X_{-1}(\CO^\vee_{\mathsf r},\triv)]_{-1}=m[\triv]_{-1}$, for some $m\in\mathbb N$. Equivalently, $\sigma(w)=m$ for all $(-1)$-elliptic $w\in W$. If $W=\bZ/2$, there is nothing to prove. Otherwise, pick two simple roots $\al,\beta$ in $\Pi$ forming an $A_2$. Then $s_\al s_\beta$ is $(-1)$-elliptic in $S_3$ and hence in $W$, which means that $\sigma(s_\al s_\beta)=\sigma(1)$. Restricting $\sigma$ to $S_3$, this means precisely that $\sigma$ does not contain any copy of the reflection representation of $S_3$. But then $\sigma$ must be the trivial or the sign representation of $W$.
\end{proof}

\begin{lemma}\label{l:subreg}
\begin{enumerate}
\item If $\CO^\vee_{\mathsf{sr}}$ is the subregular nilpotent orbit and $\rk(\fg^\vee)=2$ or $\fg^\vee$ is of type $A_3$ or $C_3$, then $\Good(\CO^\vee_{\mathsf{sr}})=\Irr W.$
\item If $W$ contains $S_5$ as a parabolic subgroup, then $\sigma\in  \Good(\CO^\vee_{\mathsf {sr}})$ only if the restriction of $\sigma$ to $S_5$ only contains the $S_5$-types $\sigma_{(5)}$, $\sigma_{(1^5)}$, $\sigma_{(41)}$, or $\sigma_{(21^3)}$.
The list of good $W$-types $\Good(\CO^\vee_{\mathsf sr})$ is:
\begin{enumerate}
\item $\{\triv,\sgn,\refl,\refl\otimes\sgn\}$ if $\fg^\vee$ is of type $A_{n-1}$, $n\ge 5$, or of type $E$;
\item$\{(n)\times \emptyset, (1^n)\times\emptyset, \emptyset\times (n),\emptyset\times (1^n), (n-1,1)\times\emptyset, \emptyset\times (2,1^{n-2}), (n-1)\times (1), (1)\times (1^{n-1})\}$, if $\fg^\vee$ is of type $C_n$, $n\ge 4$;
\item$\{(n)\times \emptyset,\emptyset\times (1^n), (n-1,1)\times\emptyset, \emptyset\times (2,1^{n-2}), (n-1)\times (1), (1)\times (1^{n-1})\}$, if $\fg^\vee$ is of type $B_n$, $n\ge 3$;
\item $\{(n)\times \emptyset, (1^n)\times\emptyset, (n-1,1)\times\emptyset,  (2,1^{n-2})\times\emptyset, (n-1)\times (1),  (1^{n-1})\times (1)\}$, if $\fg^\vee$ is of type $D_n$, $n\ge 4$;
\item $\{\phi_{1,0},\phi_{4,1},\phi_{4,13},\phi_{2,4}',\phi_{2,16}'',\phi_{1,24}\}$ if $\fg^\vee$ is of type $F_4$.
\end{enumerate}

\end{enumerate}
\end{lemma}

\begin{proof}
Claim (1) is immediate, since for $\rk(\fg^\vee)=2$ or for $A_3,C_3$, the subregular orbit is the smallest orbit in $\C N^\vee_\sol$, hence the condition for good $W$-types becomes superfluous.

For (2), suppose $\fg^\vee$ is simply laced of rank at least $4$. In that case, there is only one irreducible $W$-representation attached by the Springer correspondence to $\CO^\vee_{\mathsf{sr}}$ and that is $\refl$. As one can see from \cite{Ci,CH}, the preimage $\Theta^{-1}(\CO^\vee_{\mathsf{sr}})$ consists either of one irreducible $\wti W$-representation $\wti\sigma_{\mathsf{sr}}$ (if this is self-dual under tensoring with $\sgn$) or two representations $\{\wti\sigma_{\mathsf{sr}},\wti\sigma_{\mathsf{sr}}\otimes\sgn\}$. In addition, 
\[ \refl\otimes\C S\in \text{span} \{\wti\sigma_{\mathsf{sr}}+\wti\sigma_{\mathsf{sr}}\otimes\sgn, \C S\}.
\]
 By Proposition \ref{p:span}, we are looking for $\sigma\in\Irr W$ such that 
 \[\sigma\otimes\C S=a'(\wti\sigma_{\mathsf{sr}}+\wti\sigma_{\mathsf{sr}}\otimes\sgn)+b' \C S,\]
 for some $a',b'\in\bZ_{\ge 0}$. Equivalently, this means that there exist $a,b\in \bQ$ such that 
 \[\sigma\otimes\C S=a~ \refl\otimes\C S+ b ~\C S,
 \]
 and furthermore
 \begin{equation}\label{e:refl}
 \sigma(w)= a ~\refl(w)+ b,\text{ for all (-1)-elliptic } w\in W. 
 \end{equation}
If $\fg^\vee$ is simply laced of rank at least $4$ and $\fg^\vee\neq \mathfrak{so}(8,\bC)$, there exists a parabolic subgroup $S_5$ in $W$. Let $g_3$, $g_5$ be a $3$-cycle and a $5$-cycle in $S_5$, respectively. These are $(-1)$-elliptic elements. Using the character of the reflection representation, we see
\[\sigma(1)-\sigma(g_3)=3a\text{ and }\sigma(1)-\sigma(g_5)=5a.
\]
Restrict $\sigma$ to $S_5$: $\sigma=\sum_{\lambda\in P(5)} m_\lambda \sigma_\lambda$, where $\sigma_\lambda$ is the $S_5$-type parameterized via Young diagrams by the partition $\lambda$. Using the character table of $S_5$ which I will not reproduce here (see \cite[p. 28]{FH} for example), we find
\begin{align*}
3a=\sigma(1)-\sigma(g_3)&=3(m_{(41)}+m_{(21^3)})+6 (m_{(311)}+m_{(32)}+m_{(221)}),\\
5a=\sigma(1)-\sigma(g_5)&=5(m_{(41)}+m_{(21^3)}+m_{(311)}+m_{(32)}+m_{(221)}),
\end{align*}
hence $m_{(311)}=m_{(32)}=m_{(221)}=0$ and $a=m_{(41)}+m_{(21^3)}$.
This means that $\sigma$ can only contain upon restriction to $S_5$, the representations $\sigma_{(5)}$, $\sigma_{(1^5)}$, $\sigma_{(41)}$, and $\sigma_{(21^3)}$.

Now, suppose $\fg^\vee$ is of type $B$, $C$, or $D$. Then the claims follow easily from Lemma \ref{l:spin-B}. The difference between types $C$ and $B/D$ is the following: if $\fg^\vee=\mathfrak{sp}(2n)$, then 
\[
\Theta^{-1}(\CO^\vee_\sr)=\{((1,n-1)\times\emptyset)\otimes \C S,\ ((1^n)\times\emptyset)\otimes\C S\},
\]
but if $\fg^\vee=\mathfrak{so}(m)$, then
\[
\Theta^{-1}(\CO^\vee_\sr)=\{((1,\lfloor m/2\rfloor-1)\times\emptyset)\otimes \C S\}.
\]

For $F_4$, we use the fact that $W(F_4)$ contains $W(B_4)$ as the Weyl group of a quasi-Levi subgroup. From (\ref{e:refl}), we see that $\sigma\in \Irr W(F_4)$ is good only if all of the irreducible constituents of its restriction to $W(B_4)$ are good for $B_4$. Using Alvis' restriction tables, one can  see that the only representations of $W(F_4)$ that satisfy this condition are the ones in the statement of the lemma.

\end{proof}

\begin{proposition}\label{p:W-subsub} Suppose that $\fg^\vee$ has rank at least $4$ and it is not a symplectic or special even orthogonal Lie algebra. Let $\CO_\ssr^\vee$ be the sub-subregular nilpotent orbit in $\fg^\vee$. Explicitly, $\CO_\ssr^\vee$ is $(n-2,2)$ in $\mathfrak{sl}(n)$, $(2n-3,3,1)$ in $\mathfrak{so}(2n+1)$, $F_4(a_2)$ in $F_4$, $D_5$ in $E_6$, $E_7(a_2)$ in $E_7$, or $E_8(a_2)$ in $E_8$. Then $\Good(\CO^\vee_\ssr)\setminus \Good(\CO^\vee_\sr)$ is formed of  the following $W$-types:
\begin{enumerate}
\item[$\mathbf{A_{n-1}}$:] $\sigma_{(n-2,2)}$, $\sigma_{(n-2,1,1)}$, $\sigma_{(3,1^{n-3})}$, $\sigma_{(2,2,1^{n-4})}$, and when $n=6$, also $\sigma_{(33)}$ and $\sigma_{(222)}$;
\item[$\mathbf{B_n}$:] $(n-2)\times (2)$, $(n-2)\times (11)$, $(n-2,1)\times (1)$, $(n-2,2)\times\emptyset$, $(n-2,1,1)\times\emptyset$, $(2)\times (1^{n-2})$, $(11)\times (1^{n-2})$, $(1)\times (2,1^{n-3})$, $\emptyset\times (2,2,1^{n-4})$, $\emptyset\times (3,1^{n-3})$; 
\item[$\mathbf{F_4}$:] $\phi_{9,2}$, $\phi_{9,10}$, $\phi''_{8,3}$, $\phi'_{8,3}$, $\phi_{8,9}'$, $\phi_{8,9}''$, $\phi_{2,4}''$, $\phi_{2,16}'$;
\item[$\mathbf{E_6}$:] $\phi_{20,2}$, $\phi_{20,20}$;
\item[$\mathbf{E_7}$:] $\phi_{27,2}$, $\phi_{27,37}$, $\phi_{21,3}$, $\phi_{21,36}$;
\item[$\mathbf{E_8}$:] $\phi_{35,2}$, $\phi_{35,74}$.
\end{enumerate}
\end{proposition}

\begin{proof}

The proof uses Proposition \ref{p:span}. For an irreducible representation $\sigma$ to be good for $\CO_\ssr^\vee$, we need $[\sigma]_{-1}$ to be in the span of $[X_{-1}(\CO^\vee,\C E)]_{-1}$, where $\CO^\vee=\CO^\vee_{\mathsf r},\CO^\vee_{\sr},\CO^\vee_{\ssr}$. I analyse the explicit cases.

$\mathbf{A_{n-1}}$. If $\fg^\vee=\mathfrak{sl}(n,\bC)$, then the three relevant orbits are in Jordan form parametrisation $(n)$, $(n-1,1)$, $(n-2,2)$. The condition is then equivalent with the existence of scalars $a,b,c$ such that 
\[\sigma(w)= a~ \sigma_{(n)}(w)+b~ \sigma_{(n-1,1)}(w)+c~\sigma_{(n-2,2)}(w),\text{ for all }w\in S_n \text{ which are $(-1)$-elliptic}.
\]
Recall that an element $w$ which is $(-1)$-elliptic in a parabolic subgroup is $(-1)$-elliptic in $W$ too. Because of this, we use restrictions to smaller $S_\ell$'s. I claim that every such $\sigma\in \Irr S_n$, $n\ge 7$, has to satisfy two conditions:
\begin{enumerate}
\item[(a)] $\Hom_{S_6}[\sigma_{(3,2,1)},\sigma|_{S_6}]=0,$ and
\item[(b)] $\Hom_{S_7}[\sigma_{(4,3)},\sigma|_{S_7}]=0=\Hom_{S_7}[\sigma_{(4,1^3)},\sigma|_{S_7}]$.
\end{enumerate} 
I present the proof for (b), the one for (a) is the same, but simpler. In $S_7$, one can use the $(-1)$-elliptic elements $1$, $g_3$, $g_5$, and $g_7$ (the $7$-cycle). Using characters, it is elementary to see that in $S_7$, an irreducible representation $\mu$ is a linear combination of $\sigma_{(7)}$, $\sigma_{(6,1)}$, and $\sigma_{(5,2)}$ if and only if the system
\begin{align*}
&3(b+4c)=\mu(1)-\mu(g_3)\\
&5(b+3c)=\mu(1)-\mu(g_5)\\
&7(b+2c)=\mu(1)-\mu(g_7)\\
\end{align*}
has solutions $(b,c)$. When $\mu=\sigma_{(4,3)}$, the right hand side vector is $(15,15,14)$ and for $\mu=\sigma_{(4,1^3)}$ it is $(18,20,21)$, which does not give solutions. 

Back to $S_n$, suppose that $\sigma_\lambda$, $\lambda$ a partition of $n$ viewed as a Young diagram satisfies (a) and (b), and that $n\ge 7$ (otherwise the claim is immediate). Assume first that $\lambda=(\lambda_1\ge \lambda_2\ge\lambda_3\ge\dots)$ has at least $3$ rows. Then because of (b) and $(4,3)$ and (a) and $(3,2,1)$, we see that $\lambda_2\le 2$. If $\lambda_2=2$, using (a), we need $\lambda_1=2$. Hence, the partition is of the form $(\underbrace{2,\dots,2}_\ell,\underbrace{1\dots,1}_{n-2\ell})$. The sign dual of $\sigma$ corresponds to $(n-\ell,\ell)$. Since $\sigma$ is good if and only if $\sigma\otimes\sgn$ is, applying (b), we see that $\ell\le 2$. This means that the only possibilities are $(n)$, $(n-1,1)$, $(n-2,2)$ and their transposes.

Now if $\lambda_2=1$, then $\lambda=(n-\ell,\underbrace{1,\dots,1}_\ell)$. By applying (b) with $(4,1^3)$, we see that the only allowable partitions  have $\ell=0,1,2$, or $n-3,n-2,n-1$. The only new partitions that occur this way are $(n-2,1,1)$ and its transpose.
Finally if $\lambda$ has at most $2$ rows, using (b) with $(4,3)$, the only allowable partitions are $(n)$, $(n-1,1)$, and $(n-2,2)$.

It only remains to check that $(n-2,2)$ and $(n-1,1,1)$ are in $\Good((n-2,2))$ (then automatically their sign duals are as well, and the other allowable partitions are in $\Good((n-1,1))$ already). But this is immediate from Proposition \ref{p:span}, since $\CO^\vee_{\ssr}=(n-2,2)$ is the smallest orbit in $\C N^\vee_\sol$ whose closure contains that corresponding Springer orbits of $(n-2,2)$ and $(n-1,1,1)$.  

\

$\mathbf {B_n}$. While it is possible to prove the statement using the same technique as for $A_{n-1}$, it is easier to use Lemma \ref{l:spin-B}. The sub-subregular orbit is $(2n-3,3,1)$ and there are two local systems $\C E_0$ (the trivial) and $\C E_1$ (the sign) that enter in the Springer correspondence. The corresponding spin representations are
\[X_{-1}(\CO^\vee_\ssr,\C E_0)=((n-2,1,1)\times\emptyset)\otimes \C S\text{ and } X_{-1}(\CO^\vee_\ssr,\C E_1)=((n-2,2)\times\emptyset)\otimes \C S.
\]
Thus Proposition \ref{p:span} says that an irreducible $W(B_n)$-representation $\al\times\beta$ is good for $\CO^\vee_\ssr$ if and only if $\Ind_{S_{|\al|}\times S_{|\beta|}}^{S_n}(\sigma_\alpha\boxtimes\sigma_\beta)$ contains only $S_n$-types from the list
\[\sigma_{(n)},\ \sigma_{(1^n)},\ \sigma_{(n-1,1)},\ \sigma_{(2,1^{n-2})},\ \sigma_{(n-2,2)},\ \sigma_{(2,2,1^{n-4})},\ \sigma_{(n-2,1,1)},\ \sigma_{(3,1^{n-3})}.
\]
It is easy to see that the only possibilities are the ones listed in the statement of the lemma.

\

For exceptional groups, I verified the claim using Gap3 and the package chevie \cite{Mi}, and the character tables for the pin double covers of the exceptional Weyl groups as computed by Morris \cite{Mo-exc}.

\end{proof}

\begin{remark}\label{r:subsub}As a consequence of the proof, we see that for each of the $W$-type $\sigma$ listed in Proposition \ref{p:W-subsub}, $\Hom_{\wti W}[\sigma\otimes \C S,\wti\sigma]\neq 0$ for some $\wti\sigma\in \Theta^{-1}(\CO^\vee_\ssr)$.
\end{remark}

\section{Unitarity results}\label{s:5}
\subsection{}The Lie algebra $\fg^\vee$ is always assumed to be simple. 

\begin{lemma}\label{l:tensor-refl}
Suppose $\overline V$ is a simple $\bH$-module and let $\sigma\in \widehat W$ with $\Hom_W[\sigma,\overline V]\neq 0$. Then either $\overline V|_W=\sigma$ or there exists $\sigma'\in\widehat W$ such that $\Hom_W[\sigma',\sigma\otimes\refl]\neq 0$ and $\Hom_W[\sigma',\overline V]\neq 0$.
\end{lemma}

\begin{proof}
For every $\omega\in \fk t\subset \bH$, denote $\wti\omega=\frac 12 (\omega-\omega^*)$. The assignment $\omega\mapsto \wti\omega$ is a linear map $\fk t\to\bH$. Denote its image by $\wti{\fk t}$. As it is well known (e.g., \cite{BM2}, \cite{Op}), conjugation by $W$ defines a copy of the reflection representation on $\wti{\fk t}$. Let $v\neq 0$ be a vector in the $\sigma$-isotypic component in $\overline V$. If $\wti{\fk t}\cdot v=0$, then $\sigma$ can be lifted to a $\bH$-submodule of $\overline V$ and because of the simplicity of $\overline V$, we must have $\overline V|_W=\sigma$. Otherwise, there exists $\omega$ such that $v':=\wti\omega\cdot v\neq 0$. Since $v'$ belongs to $\sigma\otimes\refl$, this proves the claim.  
\end{proof}

The Iwahori-Matsumoto involution is defined on the generators of the graded affine Hecke algebra:
\begin{equation}
\mathsf{IM}(s_\al)=-s_\al,\ \mathsf{IM}(\omega)=-\omega,\quad \al\in \Pi,\ \omega\in\fk t.
\end{equation}
This induces an involution on the set of irreducible $\bH$-modules which maps modules with central character $W\cdot\nu$ to modules with central character $W\cdot (-\nu)$. As it is known, for every nilpotent orbit $\CO^\vee$, $h^\vee$ and $-h^\vee$ are $W$-conjugate, which means that the involution preserves the central characters $W\cdot h^\vee/2$ that considered here. Moreover, at the level of the restriction to $W$ of a $\bH$-module, $\mathsf{IM}$ corresponds to tensoring with the $\sgn$ representation. 

\begin{theorem}\label{t:unitary} Let $\overline V=\overline V(\C C,\C L)$ be a simple unitary $\bH$-module with central character $\nu$ such that $|\nu|=|h^\vee/2|$ for a nilpotent orbit $\CO^\vee$.
\begin{enumerate}
\item If $\CO^\vee=\CO^\vee_{\mathsf{r}}$, then $\nu\in W\cdot h_{\mathsf{r}}^\vee/2$ and $V$ is the trivial or the Steinberg $\bH$-module.
\item If $\rk(\fg^\vee)\ge 2$ and $\CO^\vee=\CO^\vee_{\mathsf{sr}}$, then $\nu\in W\cdot h_{\mathsf{sr}}^\vee/2$. In addition, 
\begin{enumerate}
\item if $\rk(\fg^\vee)=2$, then every simple module with this central character is unitary; 
\item if $\fg^\vee$ is $\mathfrak{sl}(4)$ or $\mathfrak{sp}(6)$, the unitary modules are given in Tables \ref{A3} and \ref{C3}.
\item if $\fg^\vee=\mathfrak{so}(7)$ or $\rk(\fg^\vee)\ge 4$, then the only unitary modules are the tempered modules $\overline V(\C C,\C L)$, $\C C\subset \CO^\vee_{\mathsf{sr}}$, and their Iwahori-Matsumoto duals, and additionally, only in the case of ${C_n}$, $G^\vee\cdot \C C=(2n-2,1,1)$  and $\C L$ trivial (where $\overline V(\C C,\C L)$ is an endpoint of complementary series), or the Iwahori-Matsumoto dual of this module.

\end{enumerate}
\item  Suppose $\fg^\vee$ is not a symplectic Lie algebra or an even orthogonal Lie algebra and that $\rk(\fg^\vee)\ge 4$. If $\CO^\vee=\CO^\vee_{\mathsf{ssr}}$ then $\nu\in W\cdot h_{\mathsf{ssr}}^\vee/2$. In addition, if $\overline V=\overline V(\C C,\C L)$ has central character $W\cdot h_{\mathsf{ssr}}^\vee/2$, then it is unitary if and only if it is a tempered module with $\C C\subset \CO^\vee_{\mathsf{ssr}}$ or its Iwahori-Matsumoto dual or:
\medskip
\begin{enumerate}
\item[$\mathbf{A_{n-1}}$:] $G^\vee\cdot \C C=(n-2,1,1)$ and $\C L$ trivial ($\overline V$ is an endpoint of complementary series), or the Iwahori-Matsumoto dual of this module, when $n\neq 6$. When $n=6$, in addition, the modules corresponding to  $G^\vee\cdot \C C=(3,3)$ or $(2,2,2)$ and $\C L$ trvial (endpoints of complementary series) are unitary. 
\item[$\mathbf{B_n}$:] $G^\vee\cdot \C C=(2n-3,2,2)$ and $\C L$ trivial, or $G^\vee\cdot \C C=(2n-3,1^4)$ and $\C L$ trivial or sign ($\overline V$ is an endpoint of complementary series in these cases), or the Iwahori-Matsumoto duals of these modules.
\item[$\mathbf{F_4}$:] $G^\vee\cdot \C C=B_3$ or $C_3$ and $\C L$ trivial ($\overline V$ is an endpoint of complementary series), or the Iwahori-Matsumoto dual of these modules.
\item[$\mathbf{E_7}$:] $G^\vee\cdot \C C=E_6$ and $\C L$ trivial ($\overline V$ is an endpoint of complementary series), or the Iwahori-Matsumoto dual of this module.
\end{enumerate}
\end{enumerate}
\end{theorem}

\begin{proof}
(1) By Lemma \ref{l:reg}, $\overline V|_W$ can only contain $\triv$ or $\sgn$. If $\rk(\fg^\vee)=1$, the only simple $\bH$-modules with central character of length $|h^\vee/2|$ are the trivial and the Steinberg, so there is nothing to prove. If $\rk(\fg^\vee)\ge 2$, Lemma \ref{l:tensor-refl} implies that $\overline V|_W=\triv$ or $\overline V|_W=\sgn$ and the claim follows.

\

(2) Since every irreducible $W$-representation $\sigma\in\Good(\CO^\vee_\sr)$, other than $\triv$ and $\sgn$, satisfies $\Hom_{\wti W}[\wti\sigma, \sigma\otimes \C S]\neq 0$ for some $\wti\sigma\in \Theta^{-1}(\CO^\vee_\sr)$, Theorem \ref{dirac-ineq} implies that the central character $\nu$ must be in $W\cdot h^\vee_\sr/2$.

If $\rk(\fg^\vee)=2$, then the possible Hecke algebras are of types $A_2$, $B_2$, $C_2$, or $G_2$. In all of these cases, one can see from the classification of the unitary duals (in the setting of the Hecke algebra $\bH$ this is not difficult for small rank root systems and it can be found in \cite{Ba}, \cite{Ci-F4}), that $h^\vee_\sr/2$ is at the boundary of the complementary series of the spherical principal series, hence all subquotients are unitary. 

If $\rk(\fg^\vee)=3$, we have the classical cases $A_3$, $B_3$, $C_3$. In the case of $A_3$, the nilpotent orbits are labelled by $(4)$, $(31)$, $(22)$, $(211)$, $(1^4)$. The subregular orbit is $(31)$ and with the standard coordinates $h^\vee_\sr/2=(1,0,0,-1)$. The classification of simple $\bH(A_3)$-modules with this central character is well known, for example, it can be found in \cite[\S 4.2]{Ci}, see Table \ref{A3} below (all local systems $\C L$ are trivial so I drop them from notation).

\begin{table}[h]\centering
\begin{tabular}{|c|c|c|c|}
\hline
$\CO^\vee$  &Label &$W$-structure &Unitary?\\
\hline
$(31)$ &$4$ &$\sigma_{(211)}+\sigma_{(1^4)}$ &Yes (tempered)\\
\hline
$(22)$ &$3$ &$\sigma_{(22)}$ &Yes (endpoint complementary series)\\
\hline
$(211)$ &$2_a$ &$\sigma_{(31)}+\sigma_{(211)}$ &No (not hermitian)\\
\hline
$(211)$ &$2_b$ &$\sigma_{(31)}+\sigma_{(211)}$ &No (not hermitian)\\
\hline
$(1^4)$ &$0$ &$\sigma_{(4)}+\sigma_{(31)}$ &Yes (spherical)\\
\hline

\end{tabular}
\smallskip
\caption{Central character $h^\vee_{\sr}/2$ for $A_3$}\label{A3}

\end{table}

In the case of $C_3$, the subregular orbit is $(42)$ and, in the standard coordinates, $h^\vee_\sr/2=(\frac 32,\frac 12,\frac 12)$. The classification of simple $\bH(C_3)$ modules with this central characters is given in \cite[\S 4.4]{Ci}, from where also the $W$-structure can be deduced easily (since we know the $W$-structure of standard modules by Springer theory and the Kazhdan-Lusztig polynomials). To check the unitarity of the simple modules, one can use the same methods as in \cite{Ci-F4}, but in fact there is only one subtle case ($3_{b,t}$ or equivalently $2_a$), where one needs to employ a signature argument with ``lowest $W$-types". The results are in Table \ref{C3}.

\begin{table}[h]\centering
\begin{tabular}{|c|c|c|c|c|}
\hline
$\CO^\vee$ &$\C E$ &Label &$W$-structure &Unitary?\\
\hline
$(42)$ &$(2)$ &$5_t$ &$1\times 11+0\times 1^3$ &Yes (tempered)\\
           &$(11)$ &$5_s$ &$1^3\times 0$ &Yes (tempered)\\
          \hline
$(411)$ &$1$ &$4_b$ &$0\times 12$ &Yes\\
\hline
$(33)$ &$1$ &$4_a$ &$11\times 1$ &Yes\\
\hline
$(222)$ &$1$ &$3_a$ &$1\times 2$ &Yes\\
\hline
$(2211)$ &$(2)$ &$3_{b,t}$ &$2\times 1+1\times 2+0\times 12+1\times 11$ &No\\
                &$(11)$ &$3_{b,s}$ &$12\times 0$ &Yes\\
                \hline
$(2211)$ &$1$ &$2_a$ &$2\times 1+12\times 0+11\times 1+1\times 11$ &No\\
\hline
$(21^4)$ &$1$ &$2_b$ &$0\times 3$ &Yes (anti-tempered)\\
\hline
$(1^6)$ &$1$ &$0$ &$3\times 0+2\times 1$ &Yes (spherical anti-tempered)\\
\hline
\end{tabular}
\smallskip
\caption{Central character $h^\vee_{\sr}/2=(\frac 32,\frac 12,\frac 12)$ for $C_3$}\label{C3}

\end{table}

\medskip

(c) If $\fg^\vee=\mathfrak{so}(7)$, the only good $W$-types for the subregular orbit are $(3)\times \emptyset$, $\emptyset\times (1^3)$, $(2)\times (1)$, $(1)\times (1,1)$, $(1,2)\times\emptyset$, $\emptyset\times (1,2)$. There are two irreducible tempered modules with this central character, $\overline V((511),\triv)|_W=(1)\times (1,1)+\emptyset\times (1^3)$ and $\overline V((511),\sgn)|_W=\emptyset\times (1,2)$ which are unitary. Their Iwahori-Matsumoto duals $\overline V(0)|_W=(3)\times \emptyset+ (2)\times (1)$ and $\overline V((221^3))=(1,2)\times \emptyset$ are then also unitary. It is easy to check that these are the only modules which contain only good $W$-types.

 Suppose $\rk(\fg^\vee)\ge 4$. If $\fg^\vee$ is of type $A$ or $E$, by Lemma \ref{l:subreg}, $\overline V$ can only contain $\triv, \refl, \sgn, \refl\otimes\sgn$. If $\overline V$ contains $\triv$, since $\overline V$ can't be the trivial module, Lemma \ref{l:tensor-refl} implies that $\overline V$ also contains $\refl$. But then, by Theorem \ref{dirac-ineq}, the central character $\nu$ is in $W\cdot h^\vee_\sr/2$. By the classification of $\bH$-modules, we know that at central character $W\cdot h^\vee_\sr/2$, the spherical module has $W$-structure $\triv+\refl$. Hence $\overline V= \overline V(0,\triv)$, the irreducible spherical module. Its Iwahori-Matsumoto dual is the simple tempered module $\overline V(\C C,\triv)$, where $G^\vee\cdot \C C=\CO^\vee_\sr$, and they are both unitary. 

The only remaining case is if $\overline V$ contains only the $W$-types $\refl$ and $\refl\otimes\sgn$. Since $\Hom_{W}[\refl,\refl\otimes\sgn]=0$, Lemma \ref{l:tensor-refl} implies that the only possibilities are $\overline V|_W=\refl$ or $\overline V|_W=\refl\otimes\sgn$. Using the Iwahori-Matsumoto involution, it is sufficient to consider the second case. By the classification of $\bH$-modules, $\overline V$ would have to be parameterised by an orbit $\C C$ such that $\C C\subset \CO^\vee_\sr$. But at this central character, there exists only one such module, the tempered one considered above. This shows there are no other possible unitary modules.

Suppose $\fg^\vee$ is of type $B_n$ or $D_n$. Then in addition to $\triv, \refl, \sgn, \refl\otimes\sgn$, the only other good $W$-types are $(n-1,1)\times \emptyset$ and its sign-dual, $\emptyset\times (2,1^{n-2})$ for $B_n$, respectively $(2,1^{n-2})\times\emptyset$ for $D_n$. In both cases, there exists a unique (unitary) $\bH$ module $\overline V$ such that $\overline V|_W=((n-1,1)\times \emptyset)\otimes \sgn$. For $B_n$, this is a tempered module $\overline V=\overline V(\C C,\sgn)$, $\C C\subset \CO^\vee_\sr$. For $D_n$, this is $\overline V(\C C,\triv)$, where $\C C\subset (2n-3,1,1,1)$ and it is an endpoint of a complementary series By the Iwahori-Matsumoto duality, their duals are also unitary. Besides these unitary modules, the only candidates are the modules that can be formed with $\triv, \refl, \sgn, \refl\otimes\sgn$, for which the analysis is exactly as for types $A$ and $E$.

Suppose $\fg^\vee$ is of type $C_n$. The tempered modules with the subregular central character are $\overline V((2n-2,2),\triv)|_W=(1)\times (1^{n-1}) +\emptyset\times (1^n)$ and $\overline V((2n-2),\sgn)|_W=(1^n)\times\emptyset$. Their Iwahori-Matsumoto duals are also unitary and have $W$-structure $(n)\times\emptyset+(n-1)\times (1)$ and $\emptyset\times (n)$, respectively. In addition, $\overline V(\C C,\triv)|_W=\emptyset \times (2,1^{n-2})$, where $\C C\subset (2n-2,1,1)$ is also unitary and an endpoint of complementary series, and so is its Iwahori-Matsumoto dual. These are the only simple modules that can be formed with the good $W$-types at this central character. 

Suppose $\fg^\vee$ is of type $F_4$. The simple modules with central character $h^\vee_\sr/2$ are listed in \cite[\S 5.2.2]{Ci} and from there, one can deduce that the only simple unitary modules that can be formed with the good $W$-types are the tempered modules 
\[ [F_4(a_1),\triv]:\ \phi_{4,13}+\phi_{1,24},\quad [F_4(a_1),\sgn]:\ \phi_{2,16}'',
\]
and their Iwahori-Matsumoto duals
\[ [0,\triv]:\ \phi_{1,0}+\phi_{4,1},\quad [A_1,\triv]:\ \phi_{2,4}'.
\]

\

(3) We are interested only in the modules that can be formed with the good $W$-types. The strategy is the following: for each good $W$-type $\sigma$, list the possible simple modules that have that $W$-type as a ``lowest $W$-type", i.e., the simple modules $\overline V(\C C,\C L)$ with central character $h^\vee_\ssr/2$ such that 
\[
\Hom_W[\sigma,\overline V(\C C,\C L)]\neq 0\text{ and } \sigma=\sigma(\CO^\vee,\C E)\text{ where }\CO^\vee=G^\vee\cdot \C C, \text{ for some }\C E.
\]
Every such module has the property that all of the $W$-types that it contains are attached via the Springer correspondence to nilpotent orbits $\CO'^\vee$ such that $\CO'^\vee\subset \overline\CO^\vee$. Moreover, we necessarily need to have $|h^\vee_{\CO^\vee}|\le  |h^\vee_{\CO^\vee_\ssr}|$. In particular, this means that $\CO^\vee$ can't be the regular or the subregular orbit. Hence, we list all the good $W$-types that are allowable, in decreasing order corresponding to the nilpotent orbits in the Springer correspondence, and the corresponding simple $\bH$-modules. 

\smallskip

$\mathbf F_4$: we can use \cite[\S 5.2.3]{Ci} to check the classification and $W$-structure of the simple modules with central character $h^\vee_\ssr/2$ (where $\CO^\vee_\ssr=F_4(a_2)$). The only ones that consist of only good $W$-types are unitary (for the reasons stated in Theorem \ref{t:unitary}) and they are (the notation is as in \cite{Ci}):
\begin{equation}
  \begin{aligned}
    &[F_4(a_2),(2)]\ (8_t): \phi_{9,10}+\phi_{4,13}+\phi_{1,24} &[F_4(a_2),(11)]\ (8_s):\ \phi_{2,16}'\\
    &[C_3,\triv]\ (7_b): \phi_{8,9}''+\phi_{2,16}'' &[B_3,\triv] (7_a): \phi_{8,9}'\\
    &[\wti A_2,\triv] (5_a): \phi_{8,3}'' &[\wti A_1,(11)] (4_{b,s}): \phi_{2,4}''\\
    &[A_1,\triv] (3): \phi_{2,4}'+\phi_{8,3}' &[0,\triv] (0): \phi_{1,0}+\phi_{4,1}+\phi_{9,2}.
    \end{aligned}
  \end{equation}

\smallskip

$\mathbf E_6$: $\phi_{20,20}$ is the lowest $W$-type of the unique tempered module with $\C C\subset \CO^\vee_\ssr=D_5$, \[\overline V(D_5,\triv)|_W=\phi_{20,20}+\phi_{6,25}+\phi_{1,36}.\] The $\mathsf{IM}$-dual is the spherical module, which is also unitary. The only other possibilities would be to have a module with lowest $W$-type $\phi_{20,2}$ which corresponds to the nilpotent $2A_1$ via Springer's correspondence or $\phi_{6,1}$ which would correspond to the nilpotent $A_1$. Moreover, any such module would have to be self-dual (or else, its dual would have appeared higher in the list). By applying the idea in Lemma \ref{l:tensor-refl}, in the first case (the second is completely analogous), we see that such a module would have to contain a ``chain'' of $W$-types $\sigma_0=\phi_{20,2}$, $\sigma_1$,\dots, $\sigma_m=\phi_{20,20}$ such that $\sigma_{i+1}$ occurs in $\sigma_i\otimes\refl$ for all $1\le i\le m-1$. But any such chain would involve $W$-types that are not good for $\CO^\vee_\ssr$. This concludes the argument for $E_6$.

\smallskip

$\mathbf E_7$: the unique tempered module with central character $h^\vee_\ssr/2$ has $W$-structure $\phi_{27,37}+\phi_{7,46}+\phi_{1,63}$. Its dual is spherical (and unitary). In addition, the irreducible module $\overline V(\C C,\triv)$, $\C C\subset E_6$ is irreducible when restricted to $W$ and it has the $W$-type $\phi_{21,36}$. It is unitary and it is an endpoint of the complementary series for the induction of the Steinberg module from $E_6$ to $E_7$. Its dual is also unitary and it consists of a single $W$-type $\phi_{21,3}$. There are no other modules that can be formed with only good $W$-types, the argument being similar to that for $E_6$. 

\smallskip

$\mathbf E_8$: this case is completely identical to $E_6$.

\smallskip

$\mathbf A_{n-1}$: suppose that $n\ge 7$ or $n=5$. The irreducible tempered module with central character $h^\vee_\ssr/2$ is \[\overline V((n-2,2))|_{W}=\sigma_{(2,2,1^{n-4})}+\sigma_{(2,1^{n-2})}+\sigma_{(1^n)}.\]
 Its dual is the spherical module. The $W$-type $\sigma_{(3,1^{n-3})}$ is the lowest $W$-type of \[\overline V((n-2,1,1))|_W=\sigma_{(3,1^{n-3})}+\sigma_{(2,1^{n-2})},\] which is unitary and a limit of the complementary series for the induced from the Steinberg on the Levi subalgebra of type $A_{n-3}$. No other modules can be formed with only good $W$-types by the same argument as in $E_6$.

  If $n=6$, then in addition to the modules found above, there are also the modules $\overline V((3,3))|_W=\sigma_{(2,2,2)}$ and $\overline V((2,2,2))|_W=\sigma_{(3,3)}$, which are both unitary.

\smallskip

$\mathbf B_n$: The sub-subregular orbit is $(2n-3,3,1)$, $n\ge 4$, in the partition notation. There are two irreducible tempered modules (in fact, discrete series modules) attached to this orbit:
\[
[(2n-3,3,1),(2)]:\ (11)\times (1^{n-2})+(1)\times (1^{n-1})+\emptyset\times (1^n),\quad [(2n-3,3,1),(11)]: \emptyset \times (2,2,1^{n-4}).
\]
The corresponding $\mathsf{IM}$-duals are
\[[0]:\ (n)\times\emptyset+(n-1)\times (1)+(n-2)\times (2),\quad [(2^4,1^{2n-8})]:\ (n-2,2)\times\emptyset.
\]
The $W$-types $(1)\times (2,1^{n-3})$ is attached in the Springer correspondence to $(2n-3,2,2)$, while $(2)\times (1^{n-2})$ and $\emptyset\times (3,1^{n-3})$ are attached  $(2n-3,1^4)$. The corresponding simple modules are:
\[[(2n-3,2,2)]:\ (1)\times (2,1^{n-3})+\emptyset \times (2,1^{n-2}),\]
\[ [(2n-3,1^4),(2)]:\ (2)\times (1^{n-2})+(1)\times (1^{n-1}),\quad [(2n-3,1^4),(11)]: \emptyset\times (3,1^{n-3}).
  \]
  They are all unitary and endpoints of complementary series. Their $\mathsf{IM}$-duals are also unitary:
  \[ [(2,2,1^{2n-4})]:\ (n-1,1)\times\emptyset+(n-2,1)\times (1),\]
  \[[(3,1^{2n-2}),(2)]: (n-1)\times (1)+(n-2)\times (11),\quad [(3,3,1^{2n-5},(11)]=(n-2,2)\times \emptyset.
  \]
  The fact that these are the only simple modules that can be formed with the good $W$-types at this central character can be proved by a similar analysis to the other cases. 
  
\end{proof}

In light of Theorem \ref{t:unitary}, the only thing left in order to complete the proof of Theorem \ref{t:main}, is the following

\begin{proposition}\label{p:gaps} Let $\overline V$ be a simple $\mathbb H$-module with central character $\nu\in\mathfrak t_\bR^\vee$.
\begin{enumerate}
\item If $\overline V$ is unitary then $|\nu|\le |h_{\mathsf{r}}^\vee/2|$.
\item If  $\rk(\fg^\vee)\ge 2$, and $|h_{\mathsf{sr}}^\vee/2|<|\nu|<|h_{\mathsf{r}}^\vee/2|$, then $\overline V$ is not unitary.
\item Suppose $\fg^\vee$ is not a symplectic Lie algebra or an even orthogonal Lie algebra and that $\rk(\fg^\vee)\ge 4$. If $|h_{\mathsf{ssr}}^\vee/2|<|\nu|<|h_{\mathsf{sr}}^\vee/2|$, then $\overline V$ is not unitary.
\end{enumerate}
\end{proposition}

\begin{proof}(1) This is immediate from the Dirac inequality, Theorem \ref{dirac-ineq}, since the maximal $d(\sigma)$ (as in definition (\ref{d:sigma})) is obtained for $\sigma=\triv$ or $\sgn$ and in that case, $d(\triv)=d(\sgn)=|h_{\mathsf{r}}^\vee/2|$. (Of course, this bound for the unitary dual is well known as noted in the introduction.)

\smallskip

(2) Suppose $\overline V$ is unitary and $|\nu|>|h_{\mathsf{sr}}^\vee/2|$. By Corollary \ref{c:good}, every $\sigma\in \Irr W$ that occurs in the restriction of $\overline V$ to $W$ must be one of the $W$-types in $\Good(\CO^\vee_{\mathsf r})=\{\triv,\sgn\}$. But then $\overline V$ is as in the proof of Theorem \ref{t:unitary}(1), and by the same argument, it must be the trivial or the Steinberg module. In particular, $\nu\in W\cdot h^\vee_{\mathsf r}/2$.

\smallskip

(3) Again, suppose $\overline V$ is unitary and $|\nu|>|h_{\mathsf{ssr}}^\vee/2|$. Under the assumptions on $\mathfrak g^\vee$, all nilpotent orbits other than $\CO^\vee_{\mathsf r}$ and $\CO^\vee_\sr$ are smaller than $\CO^\vee_\ssr$ in the closure ordering. Therefore, Corollary \ref{c:good} implies that 

\smallskip

 (i) {\it the only $W$-types that occur in $\overline V|_W$ are the ones in $\Good(\CO^\vee_\sr)$.}

\smallskip

I claim that the only modules with this property have $|\nu|\ge |h^\vee_\sr/2|$; the argument is essentially part of the proof of Theorem \ref{t:unitary}. 
Suppose that $|\nu|<|h_{\mathsf{sr}}^\vee/2|$. Let $\sigma$ be a lowest $W$-type for $\overline V$. Because of the condition on $\nu$, $\sigma$ must be attached in the Springer correspondence to a nilpotent orbit $\CO^\vee$ other than $\CO^\vee_{\mathsf r}$ and $\CO^\vee_\sr$. Moreover, since $\overline V$ is unitary if and only if $\mathsf{IM}(\overline V)$ is unitary, the same restriction holds for the lowest $W$-types of $\mathsf{IM}(\overline V)$. Recall that the $W$-types that occur in $\mathsf{IM}(\overline V)$ are the $\sgn$-dual of those in $\overline V$. So the second constraint is:

\smallskip

(ii) {\it The lowest $W$-types in $\overline V|_W$ and the lowest $W$-types in $\overline V|_W\otimes\sgn$ are not attached to the regular or subregular nilpotent orbits by Springer's correspondence.}

\smallskip

The last constraint is from Lemma \ref{l:tensor-refl} in the form used already in the proof of Theorem \ref{t:unitary}(3) (see the case of $E_6$):

\smallskip

(iii) {\it if $\sigma$ is a lowest $W$-type of $\overline V$ and $\sigma'$ is a different $W$-type occurring in $\overline V|_W$, then there must exist a chain of $W$-types $\sigma_0=\sigma$, $\sigma_1$,\dots, $\sigma_m=\sigma'$ such that $\sigma_{i+1}$ occurs in $\sigma_i\otimes\refl$ for all $1\le i\le m-1$, and all $\sigma_i$ occur in $\overline V|_W$.}

\smallskip

Using (i)--(iii), one can verify easily that there are no such modules for each type of $\mathfrak g^\vee$, $\rk(\fg^\vee)\ge 4$. To see this, partition the $W$-types from $\Good(\CO^\vee_\sr)$ into two disjoint subsets 
\begin{align*}H&=\{\sigma\in \Good(\CO^\vee_\sr)\mid \sigma \text{ is a Springer representation for }\CO_{\mathsf r}\text{ or }\CO_\sr\};\\
L&=\{\sigma\otimes\sgn\mid \sigma\in H\}.
\end{align*}
The fact that these two sets are disjoint can be seen from their explicit description: no $\sigma\in \Good(\CO^\vee_\sr)$ is $\sgn$-self-dual when $\rk(\fg^\vee)\ge 4$. Moreover, one can check immediately that
\begin{equation}\label{e:link}
\Hom_W[\sigma\otimes \refl,\sigma']=0,\text{ for all }\sigma\in L,~\sigma'\in H.
\end{equation}
Condition (ii) implies that $\overline V$ must contain at least one $W$-type $\sigma_L\in L$ and at least one $W$-type $\sigma_H\in H$. But then (\ref{e:link}) implies that condition (iii) is violated.



\end{proof}

\begin{remark}
\begin{enumerate}
\item The assumption on the Lie type of $\fg^\vee$ in Proposition \ref{p:gaps}(3) is necessary in order to have a spectral gap in the unitary dual. For example, when $\fg^\vee=\mathfrak{sp}(2n,\bC)$, the subregular nilpotent orbit is $(2n-2,2)$ and there are two nilpotent orbits immediately smaller in the closure ordering, which aren't comparable: $(2n-4,4)$ and $(2n-2,1,1)$. Since there is a complementary series from $(2n-2,1,1)$ with an endpoint at $(2n-2,2)$, it follows that there is no analogous spectral gap in this case. A similar situation occurs for even orthogonal Lie algebras.
\item The method of proof in Proposition \ref{p:gaps} using conditions (i)--(iii) could be applicable to other examples where one would like to construct simple unitary modules ``attached'' to large nilpotent orbits.
\end{enumerate}
\end{remark}

\subsection{}In the previous subsection, I used the Dirac inequality to determine unitary modules at large central characters. I now look at the other extreme.

\begin{proposition}\label{p:O-min}
There exists a unique nilpotent orbit $\CO^\vee_{\min}\in G^\vee\backslash\C N^\vee_\sol$ such that for every $\CO^\vee\in G^\vee\backslash\C N^\vee_\sol$, $\CO^\vee_{\min}\subset \overline{\CO^\vee}$. The explicit cases are:

$\mathfrak{sl}(n)$: $(1,2,\dots,k_A-1,k_A+1,\dots,\ell_A)$;

$\mathfrak{sp}(2n)$: $\CO^\vee_{\min}$ is obtained by starting with the partition $(2,2,4,4,\dots, 2\ell_C,2\ell_C)$ and removing:

\begin{itemize}
\item one $2k_C$ if  $k_C\le \ell_C$, or 
\item one $2(k_C-\ell_C)$ and one $2\ell_C$ if $k_C>\ell_C$;
\end{itemize}

$\mathfrak{so}(2n)$: start with the partition $(1,1,3,3,\dots,2\ell_D-1,2\ell_D-1)$ and remove

\begin{itemize}
\item one $1$ and one $2k_D-1$ if $k_D\le \ell_D$, or
\item one $2(k_D-\ell_D)+1$ and one $2\ell_D-1$ if $k_D>\ell_D$.
\end{itemize}

$\mathfrak{so}(2n+1)$: start with $(1,1,3,3,\dots,2\ell_C-1,2\ell_C-1,2\ell_C+1)$ and remove

\begin{itemize}
\item one $1$ and one $2 k_C-1$ if $k_C\le \ell_C$, or
\item one $2(k_C-\ell_C)-1$ and one $2\ell_C+1$ if $k_C>\ell_C$.
\end{itemize}

$G_2$: $G_2(a_1)$; 
$F_4$: $F_4(a_3)$; 
$E_6$: $D_4(a_1)$;
$E_7$: $A_4+A_1$;
$E_8$: $E_8(a_7)$;

where $\ell_A$, $\ell_B$, $\ell_C$, $\ell_D$, are the smallest integers such that:

\begin{itemize}
\item $\frac{\ell_A(\ell_A+1)}2\ge n$ and $k_A=\frac{\ell_A(\ell_A+1)}2-n$, 
\item  $\ell_C(\ell_C+1)\ge n$ and $k_C=\ell_C(\ell_C+1)-n$, 
\item $\ell_D^2\ge n$ and $k_D=\ell_D^2-n$.
\end{itemize}
\end{proposition}

\begin{proof}
For exceptional types, the claim can be checked by inspection using the tables and closure ordering in \cite{Ca}. For classical types, the verification of the claim (case by case) is elementary using the combinatorial interpretation of the closure ordering on nilpotent orbits. 

For example, for type $A_{n-1}$, a nilpotent orbit is in $\C N^\vee_\sol$ if and only if the parts of the corresponding partition are all distinct. Suppose $\lambda=(\lambda_1,\lambda_2,\dots,\lambda_m)$ is a partition of $n$ with $0<\lambda_1<\lambda_2<\dots<\lambda_m$. Set $\lambda_0=0$. If there is an $0\le i\le m-1$ such that $\lambda_i+2<\lambda_{i+1}$, then we can replace $(\lambda_i,\lambda_{i+1})$ by $(\lambda_i+1,\lambda_{i+1}-1)$ and form a smaller partition with distinct parts in the closure ordering. This means that we can assume that for all $0\le i\le m-1$, $\lambda_{i+1}-2\le \lambda_i<\lambda_{i+1}$. Now suppose there exists $1\le i\le m-1$ such that $\lambda_{i-1}+2=\lambda_i$ and $\lambda_i+2=\lambda_{i+1}$. Then we can replace $(\lambda_{i-1},\lambda_i,\lambda_{i+1})$ by $(\lambda_{i-1}+1,\lambda_i,\lambda_{i+1}-1)$ and form a partition with distinct parts which is smaller in the closure ordering. This implies that there is a unique minimal partition with distinct parts as in the statement of the proposition.
For the other classical Lie algebras, the argument is similar and I skip the details. 
\end{proof}

Clearly, for $\CO_{\min}^\vee$, the first part of the Dirac inequality does not yield any information since 
\[\Good(\CO^\vee_{\min})=\Irr W.
\]
It is therefore an interesting question to investigate the unitarity of the simple $\bH$-modules with central character precisely $h^\vee_{\min}/2$, particularly in the most ``symmetric" cases for $\CO^\vee_{\min}$:

\begin{enumerate}
\item $(1,2,3,\dots,\ell)$ in $\mathfrak{sl}(n)$, $n=\frac{\ell(\ell+1)}2$; $(2,2,4,4,\dots,2\ell,2\ell)$ in $\mathfrak{sp}(2n)$, $n=\ell(\ell+1)$; $(1,1,3,3,\dots,2\ell-1,2\ell-1)$ in $\mathfrak{so}(2n)$, $n=\ell^2$; 
\item the list for exceptional Lie types as in Proposition \ref{p:O-min}.
\end{enumerate}

\begin{remark}
If $G$ is of type $G_2$, $X(h^\vee_{\min}/2)$ lies at the boundary of the unramified complementary series \cite{Ci-F4}, hence every subquotient is unitary. The same is true for the other rank $2$ groups, but not when the rank of $G$ increases. Roughly, the unramified complementary series is at most
\[|\nu|=O(\sqrt n),
\]
whereas the bound given by $\CO^\vee_{\min}$ in the Dirac inequality is 
\[|h^\vee_{\min}/2|=O(n).
\]
\end{remark}

\begin{proposition}\label{th:F4}
Let $G=F_4$, $\CO^\vee_{\min}=F_4(a_3)$, and $\nu=h^\vee_{\min}/2$. There are $19$ non-isomorphic irreducible representations with parameter $\nu$, $18$ of which are unitary. The explicit description is given in Table \ref{F4}.
\end{proposition}

\begin{table}[h]\centering
\begin{tabular}{|c|c|c|c|c|c|}
\hline
$\CO^\vee=G^\vee\cdot \C C$ &$\C L$ &Label &$W$-structure &Unitary? &$\mathsf{IM}-dual$\\
\hline
$F_4(a_3)$ &$(4)$ &$12$ &$\phi_{12,4}+\phi_{8,9}''+\phi_{8,9}'+\phi_{9,10}+\phi_{4,15}+\phi_{1,24}$ &Yes &$[0]$\\
 &$(31)$ &&$\phi_{9,6}''+\phi_{8,9}''+\phi_{2,16}''$ &Yes &$[4]$\\
 &$(22)$ &&$\phi_{6,6}''+\phi_{4,13}$ &Yes &$[6,(2)]$\\
 &$(211)$ &&$\phi_{1,12}''$ &Yes &$[8_a,(11)]$\\
 \hline
$C_3(a_1)$ &$(2)$ &$11$ &$\phi_{16,5}+\phi_{9,10}$ &Yes &$[7_a]$ \\
 &$(11)$ &&$\phi_{4,7}''$ &Yes &$[9]$ \\
 \hline
 $A_1+\wti A_2$ &$1$ &$10_a$ &$\phi_{6,6}'$ &Yes &$[10_a]$\\
 \hline
 $\wti A_1+A_2$ &$1$ &$9$ &$\phi_{4,7}'$ &Yes &$[11,(11)]$\\
 \hline
 $B_2$ &$(2)$ &$10_b$ &$\phi_{9,6}'+\phi_{8,9}'+\phi_{2,16}'$ &Yes &$[6,(11)]$\\
 &$(11)$ &&$\phi_{4,8}$ &Yes &$[10_b,(11)]$\\
 \hline
 $A_2$ &$(2)$ &$8_a$ &$\phi_{8,3}'+\phi_{12,4}+\phi_{9,6}'+\phi_{8,9}'$ &Yes &$[8_b]$\\
 &$(11)$ & &$\phi_{1,12}'$ &Yes &$[12,(211)]$\\
 \hline
 $\wti A_2$ &$1$ &$8_b$ &$\phi_{8,3}''+\phi_{12,4}+\phi_{9,6}''+\phi_{8,9}''$ &Yes &$[8_a,(2)]$\\
 \hline
 $A_1+\wti A_1$ &$1$ &$7_a$ &$\phi_{9,2}+\phi_{16,5}$ &Yes &$[11,(2)]$\\
 \hline
 $A_1+\wti A_1$ &$1$ &$7_b$ &$\phi_{9,2}+2\phi_{16,5}+\phi_{8,3}''+\phi_{8,3}'+\phi_{8,9}''$  &No&$[7_b]$\\
 &&&$+\phi_{8,9}'+2\phi_{12,4}+\phi_{9,6}''+\phi_{9,6}'+\phi_{9,10}$&&\\
 \hline
 $\wti A_1$ &$(2)$ &$6$ &$\phi_{4,1}+\phi_{6,6}''$ &Yes &$[12,(22)]$\\
 &$(11)$ && $\phi_{2,4}''+\phi_{8,3}''+\phi_{9,6}''$ &Yes &$[10_b,(2)]$\\
 \hline
 $A_1$ &$1$ &$4$ &$\phi_{2,4}'+\phi_{8,3}'+\phi_{9,6}'$ &Yes &$[12,(31)]$\\
 \hline
 $0$ &$1$ &$0$ &$\phi_{1,0}+\phi_{4,1}+\phi_{9,2}+\phi_{8,3}''+\phi_{8,3}'+\phi_{12,4}$ &Yes &$[12,(4)]$\\
 \hline

\end{tabular}
\smallskip
\caption{Central character $h^\vee_{\min}/2$ for $F_4(a_3)$ in $F_4$}\label{F4}

\end{table}

\begin{proof}
In \cite[\S 5.1]{Ci}, the classification and multiplicities of simple modules in the standard induced modules are computed for $F_4(a_3)$. Since we know the $W$-structure of standard modules using Springer representations and parabolic induction, we can deduce from this the $W$-structure of the simple modules as listed in Table \ref{F4}. The labels in the table refer to the dimension of $\C C$.Next, for which of this simple modules, we can check the unitarisability using the explicit results and calculations in \cite{Ci-F4}. 

The first $4$ simple modules, parameterised by $F_4(a_3)$ itself, are all tempered and therefore unitary. Their Iwahori-Matsumoto duals are:
\begin{equation}
\begin{aligned}
\mathsf{IM}([12,(4)])=[0],\ \mathsf{IM}([12,(31)])=[4],\ \mathsf{IM}([12,(22)])=[6,(2)],\\ \mathsf{IM}([12,(211)])=[8_a,(11)],\\
\end{aligned}
\end{equation}
which proves the unitarisabiity of the $\mathsf {IM}$-duals.

The two simple modules parameterised by $C_3(a_1)$ are both endpoints of complementary series induced from the two subregular discrete series of $C_3$, and hence unitary. Their $\mathsf{IM}$-duals are
\begin{equation}
\mathsf{IM}([11,(2)])=[7_a],\quad \mathsf{IM}([11,(11)])=[9].
\end{equation}
The simple modules parameterised by $A_1+\wti A_2$ and $\wti A_1+A_2$ are both endpoints of complementary series of the induced from the Steinberg module on the corresponding parabolic subalgebra. Moreover, $\mathsf{IM}([10_a])=[10_a]$.

The two simple modules parameterised by $B_2$ are both endpoints of complementary series of the induced from the two subregular tempered modules on $B_3$. We record that 
\begin{equation}
\mathsf{IM}([10_b,(2)])=[6,(11)], \quad \mathsf{IM}([10_b,(11)])=[10_b,(11)],
\end{equation}
which in particular shows that the module parameterised by $\wti A_1$ and the sign local system, $[6,(11)]$ is also unitary.

The simple module parameterised by $\wti A_2$ is an endpoint of the complementary series induced from the Steinberg module on the $\wti A_2$ parabolic subalgebra. Its $\mathsf{IM}$-dual is $\mathsf{IM}([8_b])=[8_a,(2)]$, which in particular shows that the module parameterised by $[A_2,(2)]$ is also unitary.

We are only left with the module parameterised by $[7_b,1]$ which one of the modules corresponding to the nilpotent orbit $A_1+\wti A_1$. This is $\mathsf{IM}$-self-dual and it is {\it not} unitary, which can be deduced by computing the signatures of the $W$-types $\phi_{8,3}'$ and $\phi_{8,3}''$, see \cite[p. 120]{Ci-F4}.

\end{proof}

\begin{remark}
For $F_4$, the ``statistics" at central character $h^\vee/2$ are as follows:
\begin{enumerate}
\item $\CO^\vee=F_4$: \ \ \ \ \ $16$ simple modules of which $2$ are unitary;
\item  $\CO^\vee=F_4(a_1)$: $20$ simple modules of which $4$ are unitary;
\item  $\CO^\vee=F_4(a_2)$: $18$ simple modules of which $8$ are unitary;
\item  $\CO^\vee=F_4(a_3)$: $19$ simple modules of which $18$ are unitary.
\end{enumerate}
\end{remark}

\end{document}